\documentclass{article}
 \usepackage[utf8]{inputenc}
\usepackage[T1]{fontenc}
\usepackage{times}
\usepackage{amsmath}
\usepackage{amsthm}
\usepackage{amssymb}
\usepackage{mathrsfs}
\usepackage{yhmath}
\usepackage{graphicx}
\usepackage{setspace}
\usepackage{tikz}
\usepackage{tikz-cd}
\title{Local Langlands correspondence for Asai L functions and $\epsilon$-factors}
\author{Daniel Shankman}
\date{October 23, 2018}
\DeclareMathOperator{\Ker}{Ker}

\DeclareMathOperator{\GL}{GL}

\DeclareMathOperator{\Ind}{Ind}

\DeclareMathOperator{\Gal}{Gal}
\DeclareMathOperator{\der}{der}

\DeclareMathOperator{\Res}{Res}

\begin{document}

\maketitle
\section{Introduction}

\subsection{Local Langlands correspondence for general reductive groups}

Let $k$ be a local field of characteristic zero, and $W_k'$ the Weil-Deligne group of $k$.  To each finite dimensional, complex, Frobenius semisimple representation $\rho$ of $W_k'$, and each character $\psi$ of $k$, there is an associated Artin L-function $L(s,\rho)$ and $\epsilon$-factor $\epsilon(s,\rho,\psi)$, both meromorphic functions of the complex variable $s$ [Ta].  There is also the gamma factor

\[ \gamma(s,\rho,\psi) = \frac{L(1-s,\rho^{\vee})\epsilon(s,\rho,\psi)}{L(s,\rho)}\]

where $\rho^{\vee}$ is the contragredient of $\rho$.  

Let $\mathbf G$ be a connected, reductive group over $k$.  To each irreducible, admissible representation $\pi$ of $\mathbf G(k)$, each continuous, finite dimensional complex representation $r$ of the L-group $^L\mathbf G$ of $\mathbf G$ whose restriction to the connected component of $^L\mathbf G$ is complex analytic, and each character $\psi$ of $k$, there are associated a conjectural L-function $L(s,\pi,r)$ and epsilon factor $\epsilon(s,\pi,r,\psi)$.  The conjectural gamma factor $\gamma(s,\pi,r,\psi)$ is defined by
\[ \gamma(s,\pi,r,\psi) = \frac{L(1-s,\pi^{\vee},r)\epsilon(s,\pi,r,\psi)}{L(s,\pi,r)} \]
These factors are defined in many special cases, in particular by the Langlands-Shahidi method.  

The conjectural local Langlands correspondence (LLC) predicts the following:

\begin{enumerate}
    \item A partition of the classes of irreducible, admissible  representations of $\mathbf G(k)$ into finite sets, called L-packets.
    
    \item A bijection from the set of L-packets to the set of classes of admissible homomorphisms of $W_k'$ into $^L\mathbf G$ (8.2 of [Bo]).
    
    \item For each representation $r$ of $^L\mathbf G$, an equality of L and epsilon factors
    
    \[ L(s,\pi,r) = L(s,r \circ \rho) \] \[\epsilon(s,\pi,r,\psi) = \epsilon(s,r \circ \rho,\psi) \]
    whenever $\pi$ is an element of an L-packet corresponding to $\rho$, and whenever the left hand sides can be defined.
\end{enumerate}

Parts 1 and 2 of the LLC are notably established for archimedean groups [Kn], tori [Yu], and the general linear group [He1].  For archimedean groups, the left hand sides of part 3 are defined as the right hand sides.  For $\GL_n$, part 3 is established for the standard representation [He1], and for the symmetric and exterior square representations [CoShTs].

We remark that whenever the partition and bijection of parts 1 and 2 are established for $\mathbf G$, they are also established for the group $\Res_{k/k_0} \mathbf G$, where $k_0$ is a local field contained in $k$ (8.4 of [Bo]).  This procedure is compatible with the existing correspondence for archimedean groups and for tori.

\subsection{The Asai representation}

Let $E/F$ be a quadratic extension of local fields.  Let $\mathbf M$ be the group $\Res_{E/F} \GL_n$ obtained by Weil restriction of scalars.  Then $\mathbf M$ is a connected, reductive group over $F$, with $\mathbf M(F) = \GL_n(E)$.  The L-group $^L\mathbf M$ can be identified with the semidirect product of $\GL(V) \times \GL(V)$ by $\Gal(E/F)$, where $V$ is an $n$-dimensional complex vector space, and $\Gal(E/F)$ acts  by $\sigma.(T,S) = (S,T)$, where $\sigma$ is the nontrivial element of $\Gal(E/F)$.  We define the Asai representation $\mathscr R:\, ^L\mathbf M \rightarrow \GL(V \otimes V)$ by 
\[ \mathscr R(T,S) = T \otimes S \]
\[ \mathscr R(\sigma)(v \otimes v') = v' \otimes v\]
Now, let $\pi$ be an irreducible, admissible representation of $\GL_n(E)$, corresponding to the Frobenius semisimple representation $\rho: W_E' \rightarrow \GL(V)$ under the local Langlands correspondence.  As explained in 8.4 of [Bo], this corresponds to an admissible homomorphism
\[ \underline{\rho}: W_F' \rightarrow \, ^L\mathbf M \]
which we can explicitly describe as follows: identifying $W_E'$ as a subgroup of $W_F'$, choose a $z \in W_F'$ which is not in $W_E'$.  Then
\[ \underline{\rho}(a) = \begin{cases} (\rho(a),\rho(zaz^{-1}, 1_E) & \textrm{ if $a \in W_E'$} \\ (\rho(az^{-1},za,\sigma) & \textrm{ if $a \not\in W_E'$} \end{cases} \]
Thus $\pi \leftrightarrow \underline{\rho}$ is the local Langlands correspondence for $\mathbf M$.  Our main theorem, Theorem 1, is that the equality of L and epsilon factors holds for the Asai representation.

\newtheorem{Theorem1}{Theorem}

\begin{Theorem1} If $\pi$ is an irreducible, admissible representation of $\GL_n(E)$, and $\rho$ is the $n$-dimensional Frobenius semisimple representation of $W_E'$ corresponding to $\pi$, then
\[ L(s,\pi, \mathscr R) = L(s, \mathscr R \circ \underline{\rho})\] \[ \epsilon(s,\pi,\mathscr R,\psi) = \epsilon(s, \mathscr R \circ \underline{\rho},\psi)\]
\end{Theorem1}

The Asai L-function $L(s,\pi, \mathscr R)$ and epsilon factor $\epsilon(s,\pi, \mathscr R,\psi)$ can be defined using the Langlands-Shahidi method [Go].  Granting the local Langlands correspondence for archimedean groups, and its compatability with Weil restriction of scalars, Theorem 1 holds trivially in the archimedean case, because the left hand sides are equal to the right hand sides, by definition.  Thus our main interest in Theorem 1 is in the $p$-adic case.  

Asai L-functions were originally considered by T. Asai in [As].  He considered the case of an real quadratic extension $K$ of $\mathbb Q$, and associated an L-function $L(f,s)$ to a Hilbert modular form $f$ of $K/\mathbb Q$.  The Asai L-function defined above has a factorization over the places of $\mathbb Q$.  The local factor of $L(f,s)$ at the rational primes $p$ which do not split in $K$ is of the type defined above.

\section{Equality of Asai gamma factors}

Theorem 1 obviously implies the equality of the gamma factors on both sides.  However, as Henniart shows in Section 5 of [He2], the converse is also true in the $p$-adic case: if we have the equality of the Asai gamma factors, even up to a root of unity, then we can deduce the equality of the L-factors. Consequently, if we prove the equality of the gamma factors on both sides, we will have the equality of the L-functions, and then finally the equality of the $\epsilon$-factors.  Thus Theorem 1 is implied by

\newtheorem{Theorem2}[Theorem1]{Theorem}

\begin{Theorem2} If $\pi$ is an irreducible, admissible representation of $\GL_n(E)$, and $\rho$ is the $n$-dimensional Frobenius semisimple representation of $W_E'$ corresponding to $\pi$, then
\[ \gamma(s,\pi,\mathscr R,\psi) = \gamma(s, \mathscr R \circ \underline{\rho},\psi) \]
\end{Theorem2}

In fact, Henniart has already proved that the gamma factors were equal up to a root of unity using a base change argument, and so the equality of L-factors is already established.  Our method, which is similar to that of [CoShTs], gives us the exact equality of the gamma factor.

In the archimedean case, Theorem 1 and hence Theorem 2 already holds.  In the nonarchimedean case, we note that Theorem 2 holds if $\pi$ has an Iwahori fixed vector.  This follows from Proposition 8.2.7 of [Sh2] and the fact that the LLC for tori is compatible with Weil restriction of scalars.

\subsection{Tensor induction}
Let $G$ be a group, and let $H$ be an index two subgroup of $G$.  Let $z$ be a fixed element of $G$ which is not in $H$.  If $(\rho,V)$ is a finite dimensional, complex representation of $H$, define a representation of $G$ with underlying space $V \otimes V$ by
\[ g.(v\otimes v') = \begin{cases} \rho(g)v \otimes \rho(zgz^{-1})v & \textrm{ if $g \in H$} \\ \rho(gz^{-1})v' \otimes \rho(zg)v & \textrm{ if $g \not\in H$} \end{cases} \]
We call this the representation of $G$ obtained from $H$ by \emph{tensor induction}, and denote it by $\otimes \Ind_H^G \rho$.  Up to isomorphism, it does not depend on the choice of $z$.  Note that if we replace the tensor product by the direct sum, then the same definition produces the induced representation $\Ind_H^G \rho$.   

By direct calculation, we see that if $\rho$ is an $n$-dimensional representation of $W_E'$, then the representation $\mathscr R \circ \underline{\rho}$ of $W_F'$ (Section 1.2) is equal to $\otimes \Ind_{W_E'}^{W_F'} \rho$.  Thus we may restate Theorem 2 as 

\newtheorem*{Theorem2'}{Theorem 2$'$}

\begin{Theorem2'} If $\pi$ is an irreducible, admissible representation of $\GL_n(E)$, and $\rho$ is the $n$-dimensional Frobenius semisimple representation of $W_E'$ corresponding to $\pi$, then
\[ \gamma(s,\pi,\mathscr R,\psi) = \gamma(s, \otimes \Ind_{E/F}\rho,\psi) \]

where $\otimes \Ind_{E/F}\rho = \otimes \Ind_{W_E'}^{W_F'}\rho$.
\end{Theorem2'}

We will need to know how tensor induction affects a direct sum of representations.

\newtheorem{A1}{Lemma}[subsection]

\begin{A1}Let $(\rho_1,V)$ and $(\rho_2,W)$ be representations of $H$.  Define a representation $\delta$ of $G$ with underlying space $(V \otimes W) \oplus (V \otimes W)$ by

\[ \delta(g).(v\otimes w, v' \otimes w') = \begin{cases} (\rho_1(g)v \otimes \rho_2(zgz^{-1})w, \rho_1(zgz^{-1})v' \otimes \rho_2(g)w') & \textrm{ if $g \in H $} \\
(\rho_1(gz^{-1})v' \otimes \rho_2(zg)w', \rho_1(zg)v \otimes \rho_2(gz^{-1})w) & \textrm{ if $g \not\in H$} \end{cases} \]

Then $\delta \cong \Ind_H^G \rho_1 \otimes (\rho_2 \circ \iota_z)$, where $\iota_z$ denotes conjugation by $z$ ($g \mapsto zgz^{-1}$).  \end{A1}

\begin{proof}Consider the restriction of $\delta$ to $H$.  Then $\delta|_H = (\rho_1 \otimes (\rho_2 \circ \iota_z)) \oplus ((\rho_1 \circ \iota_z) \otimes \rho_2)$.  Since $z^2 \in H$, we have $\rho \circ \iota_{z^2} \cong \rho$ for any representation $\rho$ of $H$, and therefore

\begin{equation*}
    \begin{split}
        (\rho_1 \circ \iota_z) \otimes \rho_2 & \cong ((\rho_1 \circ \iota_z) \otimes \rho_2) \circ \iota_{z^2} \\
        & = ((\rho_1 \circ \iota_z) \otimes \rho_2) \circ \iota_z \circ \iota_z \\
        & = ((\rho_1 \circ \iota_{z^2}) \otimes (\rho_2 \circ \iota_z)) \circ \iota_z \\
        & \cong (\rho_1 \otimes (\rho_2 \circ \iota_z)) \circ \iota_z
    \end{split}
\end{equation*}

This implies that 

\[ \delta|_H \cong (\rho_1 \otimes \rho_2 \circ \iota_z) \oplus (\rho_1 \otimes \rho_2 \circ \iota_z) \circ \iota_z \]

which is exactly the restriction of $\Ind_H^G\rho_1 \otimes (\rho_2 \circ \iota_z)$ to $H$ (given in the remark after the definition of tensor induction above).  It is straightforward to verify that the composition of these isomorphisms actually intertwines the action of $G$, not just $H$.  

\end{proof}

If $V$ is a complex vector space which is equal to a direct sum $V_1 \oplus \cdots \oplus V_r$, then $V \otimes V$ is equal to a direct sum

\[ [\bigoplus\limits_{i=1}^r V_i \otimes V_i] \oplus [ \bigoplus\limits_{1 \leq i < j \leq r} (V_i \otimes V_j) \oplus (V_i \otimes V_j)] \]

The same procedure allows us to decompose a representation obtained by tensor induction:

\newtheorem{A2}[A1]{Lemma}

\begin{A2} Suppose that $(\rho,V)$ is a representation of $H$ which decomposes as a direct sum of subrepresentations $(\rho_1,V_1) \oplus \cdots \oplus (\rho_r, V_r)$.  Then $(\otimes \Ind_H^G \rho, V \otimes V)$ decomposes as a direct sum of subrepresentations

\[ \otimes \Ind_H^G \rho = \bigoplus\limits_{i=1}^r \otimes \Ind_H^G \rho_i \oplus \bigoplus\limits_{1 \leq i < j \leq r} \Ind_H^G \rho_i \otimes (\rho_j \circ \iota_z) \]

\end{A2}

\begin{proof} This follows from direct computation and applying Lemma 2.1.1.  \end{proof}

If the representation $\rho$ is not semisimple, one can replace the direct sum by a composition series and obtain a similar result.  The proof is similar to Lemma 2.1.2.

\newtheorem{A3}[A1]{Lemma}

\begin{A3}Suppose that $(\rho,V)$ is a finite dimensional representation of $H$.  Let $(V_1,\rho_1)$, ... , $(V_r,\rho_r)$ be the composition factors of a composition series of $\rho$.  There is a sequence of subrepresentations of $\otimes \Ind_H^G \rho$:

\[ 0 = L_0 \subset L_1 \subseteq \cdots \subset V \otimes V \]

for which the following representations show up as the quotients $L_i/L_{i+1}$:

\[ \otimes \Ind_H^G \rho_i : 1 \leq i \leq r \] 

\[ \Ind_H^G \rho_i \otimes (\rho_j \circ \iota_z) : 1 \leq i < j \leq r \]  \end{A3}

Suppose that $\eta$ is a character of $H$.  Then $\otimes \Ind_H^G \eta$ is a character of $G$, and it is immediate that tensor induction commutes with twisting by characters: 
\[(\otimes \Ind_H^G \rho)(\otimes \Ind_H^G \eta) = \otimes \Ind_H^G (\rho \eta) \]

\newtheorem{A4}[A1]{Lemma}

\begin{A4} Let $E/F$ be a quadratic extension of $p$-adic fields.  Let $\rho$ be a character of the Weil group $W_E$.

(i): If $\rho$ is an unramified character of $W_E$, then $\otimes \Ind_{E/F} \rho$ is an unramified character of $W_F$.  More specifically, if $|| \cdot ||$ is the Weil group norm, and $\rho = || \cdot ||^{s_0}$ for a complex number $s_0$, then $\otimes \Ind_{E/F} \rho = || \cdot ||^{2s_0}$.

(ii): The character $\otimes \Ind_{E/F} \rho$ can be made highly ramified by choosing $\rho$ to be highly ramified.

\end{A4}

\begin{proof} (i) is a straightforward computation, and (ii) is a short analytic argument which can be seen by identifying $W_F^{\operatorname{ab}}$ and $W_E^{\operatorname{ab}}$ with $F^{\ast}$ and $E^{\ast}$ via local class field theory.  One uses the fact that the homomorphism $W_E^{\operatorname{ab}} \rightarrow W_F^{\operatorname{ab}}$ coming from the inclusion map $W_E \subset W_F$ identifies with the norm $E^{\ast} \rightarrow F^{\ast}$.  
\end{proof}

\subsection{Tensor induction for the global Weil group}

Let $K/k$ be a quadratic extension of number fields.  Fix an algebraic closure $\overline{k}$ of $k$ containing $K$, and a global Weil group $\varphi: W_k \rightarrow \Gal(\overline{k}/k)$.  For each place $v$ of $k$, fix an algebraic closure $\overline{k_v}$ of the completion $k_v$, and an embedding $i_v: \overline{k} \rightarrow \overline{k_v}$ over $k$, which induces an injection $\Gal(\overline{k_v}/k_v) \rightarrow \Gal(\overline{k}/k)$ defined by $\tau \mapsto i_v^{-1} \circ \tau \circ i_v$.

For each place $v$ of $k$, we have a local Weil group $W_{k_v} \rightarrow \Gal(\overline{k_v}/k_v)$, related to the global Weil group by the following lemma:

\newtheorem{B1}{Lemma}[subsection]

\begin{B1}There exists a continuous homomorphism $\theta_v: W_{k_v} \rightarrow W_k$, unique up to inner isomorphism by an element of $\Ker \varphi$, such that the diagram

\[ \begin{tikzcd} W_k \arrow[r,"\varphi"] & \Gal(\overline{k}/k) \\ W_{k_v} \arrow[r]  \arrow[u,"\theta_v"]& \Gal(\overline{k_v}/k_v) \arrow[u,] \end{tikzcd} \]

is commutative.  \end{B1}

\begin{proof} Proposition 1.6.1 of [Ta].
\end{proof}

We will carry over Lemma 2.2.1 in a compatible manner to the field $K$.  We take the global Weil group $W_K$ for $K$ to be the preimage under $\varphi$ of $\Gal(\overline{k}/K)$.  For each place $v$ of $k$, the embedding $i_v: \overline{k} \rightarrow \overline{k_v}$ determines a place $w$ of $K$ lying over $v$: the completion $K \rightarrow K_w$ is obtained by taking $K_w$ to be the composite field $i_v(K)k_v$.

For each place $v$ of $k$, and the determined place $w$ of $K$ lying over $v$, we may take the local Weil group $W_{K_w}$ to be the preimage of $\Gal(\overline{k_v}/K_w)$ under $W_{k_v} \rightarrow \Gal(\overline{k_v}/k_v)$.  Then we have:  

\newtheorem{B2}[B1]{Lemma}

\begin{B2} Let $w$ be the place of $K$ determined by a place $v$ of $k$.  Take $\theta_w$ be the restriction of $\theta_v$ to $W_{K_w}$.  Then $\theta_w: W_{K_w} \rightarrow W_K$ is the continuous homomorphism, unique up to inner isomorphism by an element of $\Ker \varphi|_{W_K}$, such that the diagram

\[\begin{tikzcd} W_K \arrow[r,"\varphi"] & \Gal(\overline{k}/K) \\ W_{K_{w}} \arrow[r]  \arrow[u,"\theta_{w}"]& \Gal(\overline{k_v}/K_{w}) \arrow[u,] \end{tikzcd} \] 
commutes.  
\end{B2}

Next, choose once and for all an element $Z$ in $W_k$ which is not in $W_K$, and let $\tilde{\sigma}$ be the image of $Z$ in $\Gal(\overline{k}/k)$.  

Suppose that $w$ is not the only place of $K$ which lies over a given place $v$.  Let $w'$ be the other one.  Then the map $i_{w'} = i_v \circ \tilde{\sigma}^{-1}: \overline{k} \rightarrow \overline{k_v}$ gives an embedding of algebraic closures of $K, K_{w'}$ respectively, through which we obtain a homomorphism $\Gal(\overline{k_v}/K_{w'}) \rightarrow \Gal(\overline{k}/K)$.    

Then we may take $\theta_{w'}: W_{K_{w'}} \rightarrow W_K$ to be the homomorphism given by $\theta_{w'}(x) = Z \theta_w(x)Z^{-1}$.  We see immediately that the diagram

\[ \begin{tikzcd} W_K \arrow[r,"\varphi"] & \Gal(\overline{k}/K) \\ W_{K_{w'}} \arrow[r]  \arrow[u,"\theta_{w'}"]& \Gal(\overline{k_v}/K_{w_1}) \arrow[u,] \end{tikzcd} \]

is commutative.  Of course, under our identifications, $K_w = K_{w_1} = k_v$, and $W_{K_w} = W_{K_{w_1}} = W_{k_v}$.

The point of doing all of this is that in terms of the fixed homomorphisms $\theta_v : v \textrm{ a place of } k$, and the choice of $Z \in W_k - W_K$, we now have for \emph{every} place of $K$ a homomorphism $\theta_w: W_{K_w} \rightarrow W_K$ as in Lemma 2.2.2.  

Finally, suppose that $(\Sigma,V)$ is a continuous, finite dimensional, complex representation of $W_K$.  We will consider the representation $\otimes \Ind_{K/k} \Sigma$ of $W_k$ obtained from $\Sigma$ by tensor induction. 

\bigskip 

If $F$ is a function on $W_K$ and $w$ is a place of $K$, we define $F_w$ to be $F \circ \theta_w$, the ``restriction of $F$ to $W_{K_w}$.''  The same for a function on $W_k$ and a place of $k$.  

\newtheorem{B3}[B1]{Proposition}

\begin{B3} Let $v$ be a place of $k$, and let $w$ be the place of $K$ lying over $v$ which is determined by the embedding $i_v$.  

(i): If $w$ is the only place of $k$ lying over $v$, then
\[( \otimes \Ind_{K/k} \Sigma)_v = \otimes \Ind_{K_w/k_v} \Sigma_w \]

(ii): If there is another place $w'$ of $K$ lying over $v$, then 
\[( \otimes \Ind_{K/k} \Sigma)_v = \Sigma_w \otimes \Sigma_{w'} \]
\end{B3}

\begin{proof} (i): If $a \in W_{k_v}$, then $\theta_v(a)$ lies in $W_K$ if and only if $a$ is in $W_{K_w}$.  Consequently, $(\otimes \Ind_{K/k} \Sigma)_v$ is given for $a \in W_{k_v}$ by

\[ a.( v\otimes v') = \begin{cases} \Sigma(\theta_w(a))v \otimes \Sigma(Z\theta_w(a)Z^{-1})v' & \textrm{ if $a \in W_{K_w}$} \\
\Sigma(\theta_v(a)Z^{-1})v' \otimes \Sigma(Z\theta_v(a))v & \textrm{ if $a \not\in W_{K_w}$}\end{cases} \]

Let $z$ be any element of $W_{k_v}$ which is not in $W_{K_w}$.  Then $\theta_v(z)$ is an element of $W_k$ which is not in $W_K$.  Consequently, $\theta_v(z)Z^{-1}$ lies in $W_K$, and the map $v \otimes v' \mapsto v \otimes \Sigma(\theta_v(z)Z^{-1})v'$ defines an isomorphism

\[ (\otimes \Ind_{K/k}\Sigma)_v \xrightarrow{\cong} \otimes \Ind_{K_w/k_v} \Sigma_w \]

(ii): The image of $\theta_v$ is contained in $W_K$, and the representation $(\otimes \Ind_{K/k} \Sigma)_v$ of $W_{k_v}$ is given by

\[ a.( v \otimes v') = \Sigma(\theta_v(a)) v \otimes \Sigma(Z \theta_v(a)Z^{-1})v' \]

for all $a \in W_{k_v}$.  But recall that under our identifications, $W_{k_v} = W_{K_w} = W_{K_{w'}},$ $\theta_v = \theta_w$, and $\iota_Z \circ \theta_v = \theta_{w'}$, where $\iota_Z$ denotes conjugation by $Z$.  Therefore, we have

\[ (\otimes \Ind_{K/k}\Sigma)_v = \Sigma_w \otimes \Sigma_{w'} \]

whenever there are two distinct places $w$ and $w'$ lying over $v$.

\end{proof}

\subsection{Multiplicativity of the analytic gamma factor}

Let $\mathbf G$ be a quasi-split group over a characteristic zero local field $k$, and let $\pi$ be an irreducible, admissible representation of $\mathbf G(k)$.  Let $r: \, ^L\mathbf G \rightarrow \GL(V)$ be an irreducible representation for which $L(s,\pi,r)$ and $\epsilon(s,\pi,r)$ can be defined by the Langlands-Shahidi method.  We will call such a representation an \emph{LS-representation}.  

The principle of multiplicativity, which is crucial to the Langlands-Shahidi method, can be stated in the following way.  Suppose that $\mathbf P$ is a parabolic $k$-subgroup of $\mathbf G$ with Levi factor $\mathbf G_{\ast}$, and $\pi$ is isomorphic to a subrepresentation of $I_{\mathbf G_{\ast}}^{\mathbf G} \pi_{\ast}$, where $\pi_{\ast}$ is an irreducible, admissible representation of $\mathbf G_{\ast}(k)$, and $I_{\mathbf G_{\ast}}^{\mathbf G}$ denotes normalized parabolic induction.  Let $r_1, ... , r_t$ be the irreducible constituents of the restriction of $r$ to $^L\mathbf G_{\ast}$.  Then each $r_i$ is an LS-representation, and we have
\[ \gamma(s,\pi,r,\psi) = \prod\limits_{i=1}^t \gamma(s,\pi_{\ast},r_i,\psi) \]

\newtheorem{C}{Proposition}[subsection]

\begin{C} Let $\pi$ be a smooth, irreducible representation of $\GL_n(E)$, and suppose $\pi$ is isomorphic to a subrepresentation of $\Ind^{\GL_n(E)} \pi_1 \boxtimes \cdots \boxtimes \pi_r$, where $\pi_i$ is a smooth, irreducible representation of $\GL_{n_i}(E)$.  Then $\gamma(s,\pi, \mathscr R, \psi)$ is equal to 

\[ \prod\limits_{i=1}^r \gamma(s,\pi_i, \mathscr R, \psi) \prod\limits_{1 \leq i < j \leq r} \lambda(E/F,\psi)^{n_in_j} \gamma(s, \pi_i \times (\pi_j \circ \sigma), \psi \circ  \operatorname{Tr}_{E/F}) \]

where $\lambda(E/F,\psi)$ is the Langlands lambda function (8.2 of [Sh2]).
\end{C}

\begin{proof} In Section 5 of [Go], Goldberg has established this multiplicativity, up to a root of unity.  The necessity of the Langlands lambda function to make the exact equality follows from the fact that the Rankin product is being computed in the setting of restriction of scalars and the following general principle:

Let $k/k_0$ be a finite extension of $p$-adic fields.  Let $\mathbf G$ be a split group over $k$ containing a Borel subgroup $\mathbf B = \mathbf T \mathbf U$ and maximal standard parabolic subgroup $\mathbf P = \mathbf M \mathbf N$, and let $\mathbf G_0 = \Res_{k/k_0} \mathbf G$ (similarly $\mathbf N_0$, $\mathbf M_0$, etc.).  Let $\pi$ be a smooth, irreducible representation of $\mathbf M(k) = \mathbf M_0(k_0)$. For $\psi$ a character of $k_0$, and $\psi_0 = \psi \circ \operatorname{Tr}_{k/k_0}$ a character of $k$, a  character $\chi$ of $\mathbf U(k) = \mathbf U_0(k_0)$ is $\psi$ (resp. $\psi_0$) generic (3.1 of [Sh2])  whether it is considered as a character of $\mathbf U(k)$ or of $\mathbf U_0(k_0)$.  Furthermore, if $\pi$ is $\chi$-generic, the Shahidi local coefficient [Sh1] $ C_{\chi}(s,\pi)$ is the same whether it is interpreted as coming from $\mathbf M$ in $\mathbf G$ or from $\mathbf M_0$ in $\mathbf G_0$.  All of this follows readily from the definitions.
    
Next, suppose that the adjoint action $r$ of $^L\mathbf M$ on $\operatorname{Lie}(^L\mathbf N)$ is irreducible.  Then so is the adjoint action $r_0$ of $^L\mathbf M_0$ on $\operatorname{Lie}(^L\mathbf N_0)$.  Shahidi's theorem (8.3.2 of [Sh2]) says that
    
\[ \lambda(k/k_0,\psi)^{\operatorname{Dim} \mathbf N} \gamma(s,\pi,r_0^{\vee},\overline{\psi_0}) = C_{\chi}(s,\pi) = \gamma(s,\pi,r^{\vee},\overline{\psi})\]
    
Note also that the appearance of the $\pi_j \circ \sigma$ in the Rankin product follows from the following general principle:

Let $\mathbf M_0$ be a connected, reductive group over $F$, let $\pi$ be a smooth, irreducible representation of $\mathbf M_0(F)$, and let $r$ be an LS-representation of $^L\mathbf M_0$.  If $\varphi$ is an automorphism of $\mathbf M_0$ over $F$, then  $\varphi$ induces an automorphism $^L\varphi$ of $^L\mathbf M_0$.  Then $r \circ \, ^L\varphi$ is an LS-representation of $^L\mathbf M_0$, and we have
\[ \gamma(s,\pi, r,\psi) = \gamma(s, \pi \circ \varphi, r \circ \,  ^L\varphi,\psi)\]
We then take $\mathbf M_0 = \Res_{E/F} \GL_n \times \GL_m$, and $\varphi$ the automorphism of $\mathbf M_0$ defined on points by $(x,y) \mapsto (x, \sigma(y))$.  \end{proof}

Suppose that $\pi_1$ and $\pi_2$ are representations of $\GL_n(E)$ and $\GL_m(E)$ corresponding to Frobenius semisimple representations $\rho_1$ and $\rho_2$ of $W_E'$.  Then $\pi_2 \circ \sigma$ corresponds to the representation $\rho_2 \circ \iota_z$, where $\iota_z$ is conjugation by an element $z$ of $W_F'$ which is not in $W_E'$, and we have

\begin{equation*}
    \begin{split}
         \gamma(s, \pi_1 \times (\pi_2 \circ \sigma), \psi \circ \operatorname{Tr}_{E/F}) & =  \gamma(s, \rho_1 \otimes (\rho_2 \circ \iota_z), \psi \circ \operatorname{Tr}_{E/F}) \\
        & = \lambda(E/F,\psi)^{-nm}\gamma(s, \Ind_{E/F}\rho_1 \otimes (\rho_2 \circ \iota_z), \psi)
    \end{split}
\end{equation*}

\subsection{Template of the global argument}

We return to the notation and conventions of Section 2.2.  Let $\Sigma$ be an $n$-dimensional representation of the global Weil group $W_K$, and let $\Psi = \otimes \psi_v$ be a character of $\mathbb A_k/k$.  Consider the representation $\otimes \Ind_{K/k} \Sigma$ of $W_k$ obtained from $\Sigma$ by tensor induction, and the global L-function and epsilon factor

\[ L(s,\otimes \Ind_{K/k} \Sigma) = \prod\limits_v L(s,(\otimes \Ind_{K/k} \Sigma)_v)\]
\[ \epsilon(s, \otimes \Ind_{K/k} \Sigma) = \prod\limits_v \epsilon(s, (\otimes \Ind_{K/k} \Sigma)_v, \Psi_v) \]

From the global functional equation 
\[ L(s,\otimes \Ind_{K/k} \Sigma) = \epsilon(s,\otimes \Ind_{K/k} \Sigma)L(1-s, (\otimes \Ind_{K/k} \Sigma)^{\vee}) \]

we get
\[ \prod\limits_v \gamma(s, (\otimes \Ind_{K/k} \Sigma)_v,\Psi_v) =1 \]

Let $\Pi_w$ be the representation of $\GL_n(K_w)$ corresponding to the semisimplification of $\Sigma_w$ under the local Langlands correspondence.  Suppose that $\Pi = \otimes \Pi_w$ is a cuspidal automorphic representation of $\GL_n(\mathbb A_K)$.  We will let $\mathbf H = \Res_{K/k} \GL_n$, and consider $\Pi = \otimes_v \pi_v$ as a cuspidal automorphic representation of $\mathbf H(\mathbb A_k)$.  The L-group of $\mathbf H$ is the semidirect product of $\GL(V) \times \GL(V)$ by $\Gal(K/k)$, where $V$ is an $n$-dimensional complex vector space, and we can define the global Asai representation $R$ of $^L\mathbf H$ exactly as in the local case.  We have the global L-function and epsilon factor

\[ L(s,\Pi,R) = \prod\limits_v L(s,\pi_v, R_v)\]
\[ \epsilon(s,\Pi,R) = \prod\limits_v \epsilon(s,\pi_v,R_v,\Psi_v)\]

where $R_v$ is the composition of $^L\mathbf H_{k_v} \rightarrow \, ^L\mathbf H$ and $R$.  From the global functional equation

\[ L(s,\Pi,R) = \epsilon(s,\Pi,R) L(1-s,\Pi^{\vee},R)\]

we get

\[ \prod\limits_v \gamma(s,\pi_v,R_v,\Psi_v) = 1\]

Therefore,

\[ \prod\limits_v \gamma(s,\pi_v,R_v,\Psi_v) = \prod\limits_v \gamma(s, (\otimes \Ind_{K/k} \Sigma)_v,\Psi_v) \]

\newtheorem{D}{Proposition}[subsection]

\begin{D} If $v$ is a place of $k$, and any of the following conditions is met:

(i): There are two places $w$ and $w'$ of $K$ lying over $v$

(ii): $v$ is archimedean

(iii): $v$ is nonarchimedean and $\pi_v$ has an Iwahori fixed vector

then $\gamma(s,\pi_v,R_v,\Psi_v) = \gamma(s, (\otimes \Ind_{K/k} \Sigma)_v,\Psi_v)$.

\end{D}

\begin{proof} Assume we are in the case (i).  Then we can identify $k_v = K_w = K_{w'}$, and we have

\[ \mathbf H_{k_v} = \GL_n \times \GL_n\]
\[ ^L\mathbf H_{k_v} = \GL(V) \times \GL(V)\]
\[ \pi_v = \Pi_w \boxtimes \Pi_{w'} \]
\[ \gamma(s, \pi_v, R_v,\Psi_v) = \gamma(s, \Pi_w \times \Pi_{w'},\Psi_v)\]
\[ (\otimes \Ind_{K/k} \Sigma)_v = \Sigma_w \otimes \Sigma_{w'} \]

where this last equality is Proposition 2.2.3.  We then have the equality of gamma factors, because tensor products go to Rankin products under the local Langlands correspondence.  Note that the gamma factor on the Artin side only depends on the semisimplification of the representation.  

Assume that we are not in the case (i), so there is only one place $w$ of $k$ lying over $v$.  Then

\[ \mathbf H_{k_v} = \Res_{K_w/k_v} \GL_n \]
\[ \pi_v = \Pi_w \]
\[ R_v = \mathscr R\]
\[ (\otimes \Ind_{K/k} \Sigma)_v = \otimes \Ind_{K_w/k_v} \Sigma_w\]
where the last equality is again Proposition 2.2.3.  The equality of gamma factors is then exactly the assertion of Theorem 2$'$, which we have remarked is valid in the cases (ii) and (iii) (see the remark below the statement of Theorem 2).  
\end{proof}

Proposition 2.4.1 allows us to cancel off all but finitely many of the gamma factors, leaving us with

\[ \prod\limits_{v \in S} \gamma(s,\pi_v,\mathscr R,\Psi_v) = \prod\limits_{v \in S} \gamma(s, \otimes \Ind_{K_w/k_v} \Sigma_w,\Psi_v) \tag{2.4}\]

where $S$ is a finite set of finite places $v$ of $k$, each of which has only one place $w$ of $K$ lying over it.

\subsection{The stable equality}

From now on, $E/F$ will be a quadratic extension of $p$-adic fields.  Let $\pi$ be a smooth, irreducible representation of $\GL_n(E)$, corresponding to an $n$-dimensional Frobenius semisimple representation $\rho$ of $W_E'$.  Let $\eta$ be a character of $W_E$, which we identify as a character of $E^{\ast}$ by local class field theory.  Then define $\pi \eta$ to be $\pi(\eta \circ \operatorname{det})$.  Then $\pi \eta$ corresponds to $\rho \eta$ under the LLC.  

\newtheorem{D1}{Lemma}[subsection]

\begin{D1} Let $\eta = || \cdot ||^{s_0}$ be an unramified character of $W_E$, and $\pi$ a smooth, irreducible representation of $\GL_n(E)$.  Then

\[ \gamma(s, \pi \eta, \mathscr R,\psi) = \gamma(s+2s_0,\pi,\mathscr R,\psi)\]
\[ \gamma(s,\otimes \Ind_{E/F}(\rho \eta),\psi) = \gamma(s+2s_0, \otimes \Ind_{E/F} \rho,\psi) \]

\end{D1}

\begin{proof} The second equality follows from Lemma 2.1.4 (i), and general facts about how Artin factors are affected by unramified twists.  The first equality will be proved later in the generic case in Lemma 3.3.1.  When $\pi$ is not generic, one reduces to the generic case using multiplicativity (Proposition 2.3.1).  \end{proof}

It follows from Lemma 2.5.1 that if Theorem $2'$ holds for a given representation $\pi$, it also holds for unramified twists of $\pi$.

In this section, we prove the following stable version of Theorem 2$'$ for supercuspidal representations:

\newtheorem{D2}[D1]{Proposition}

\begin{D2} (Stable equality for supercuspidals) Let $\pi$ be a supercuspidal representation of $\GL_n(E)$, corresponding to an irreducible representation $\rho$ of $W_E$.  Then for all sufficiently highly ramified characters $\eta$ of $W_E$, we have
\[ \gamma(s, \pi \eta, \mathscr R, \psi) = \gamma(s, \otimes \Ind_{E/F}(\rho \eta), \psi) \]

\end{D2}

We first prove Proposition 2.5.2 in the case $n = 1$.  For this, we require a simple lemma on idelic characters:

\newtheorem{D3}[D1]{Lemma}

\begin{D3} Let $K$ be a number field, $w_1, ... , w_r$ finite places of $K$, and  $\eta_1, ... , \eta_r$ characters of $K_{w_1}^{\ast}, ... , K_{w_r}^{\ast}$.  There exists a unitary character $\mathscr X = \otimes_w \mathscr X_w$ of $\mathbb A_K^{\ast}/K^{\ast}$ such that $\mathscr X_{w_i}$ and $\eta_i$ agree on $\mathcal O_{w_i}^{\ast}$ for $1 \leq i \leq r$, and $\mathscr X_w$ is unramified at all other finite places $w$. 

\end{D3}

\begin{proof} Consider the compact subgroup $\prod\limits_{w < \infty} \mathcal O_w^{\ast}$ of $\mathbb A_K^{\ast}$.  Define a unitary character $\mathscr X = \otimes \mathscr X_w$ of this compact subgroup by setting $\mathscr X_{w_i} = \eta_i$ for $1 \leq i \leq r$, and $\mathscr X_w = 1$ for all other finite places $w$.  Since this compact subgroup has trivial intersection with $K^{\ast}$, we can extend $\mathscr X$ to a unitary character of $K^{\ast} \prod\limits_{w < \infty} \mathcal O_w^{\ast}$ by making it trivial on $K^{\ast}$.  Now $K^{\ast} \prod\limits_{w < \infty} \mathcal O_w^{\ast}$ is closed $\mathbb A_K^{\ast}$, being the product of a closed set and a compact set.  Hence $\mathscr X$ can be extended to a unitary character of $\mathbb A_K^{\ast}$ by Pontryagin duality.  It is trivial on $K^{\ast}$, and satisfies the requirements of the lemma.

\end{proof}

\newtheorem{D4}[D1]{Lemma}

\begin{D4} Theorem 2$'$ (and hence Proposition 2.5.2) holds for the case $n = 1$.

\end{D4}

Note that Proposition 2.5.2 follows from Theorem 2$'$, because the local Langlands correspondence is compatible with twisting by characters.  

\begin{proof} In this case, $\pi$ is a character $\chi$ of $E^{\ast} = \Res_{E/F} \GL_1(F)$, and the $\rho$ corresponding to $\chi$ is a character of $W_E$.   We globalize the situation by finding:

\begin{enumerate}
    \item A quadratic extension of number fields $K/k$, with places $w_0 \mid v_0$, such that $K_{w_0} = E, k_{v_0} = F$
    
    \item A unitary character $\mathscr X = \otimes \mathscr X_w$ of $\mathbb A_K/K^{\ast}$ such that $\mathscr X_{w_0}$ agrees with $\chi$ on $\mathcal O_E^{\ast}$, and $\mathscr X_w$ is unramified for finite $w \neq w_0$ (Lemma 2.5.3).
    
    \item A character $\Psi = \otimes \Psi_v$ of $\mathbb A_k/k$ such that $\Psi_{v_0} = \psi$.
\end{enumerate}

We identify $\mathscr X$ as a one dimensional representation of the global Weil group, so that for each place $w$ of $K$, the one dimensional representation $(\mathscr X)_w$ of $W_{K_w}$ identifies with the character of $\mathscr X_w$ of $K_w^{\ast}$.  Now we identify $\mathscr X = \otimes_v \mathfrak X_v$ as a cuspidal automorphic representation of $\Res_{K/k} \GL_1(\mathbb A_k)$.  Using the template of the global argument (Section 2.4), we have
\[ \prod\limits_{v \in S} \gamma(s, \mathscr X_w, \mathscr R, \Psi_v) = \prod\limits_{v \in S} \gamma(s, \otimes \Ind_{K_w/k_v} \mathscr X_w, \Psi_v) \]
where $S$ is a finite set of finite places containing $v_0$, such that each place $v$ of $S$ has only one place $w$ of $K$ lying over it.  But since $\mathscr X_w$ is unramified at every place other than $w_0$, the situation (iii) of the Proposition 2.4.1 applies to allow us to cancel off all the gamma factors other than at $v_0$, giving us
\[ \gamma(s, \mathscr X_{w_0}, \mathscr R, \psi) = \gamma(s, \otimes \Ind_{E/F} \mathscr X_{w_0},\psi) \]
Now $\pi = \chi$ agrees with $\mathscr X_{w_0}$, hence they differ by an unramified character $|| \cdot ||^{s_0}$ for some complex number $s_0$.  Since both gamma factors are compatible with twisting by unramified characters (Lemma 2.5.1), we get the required equality.  
\end{proof}

Now, we will finish the proof of the stability equality theorem by induction on $n$.  Let us state formally our induction hypothesis:

\newtheorem*{IH}{Induction Hypothesis}

\begin{IH} Let $n \geq 2$.  For each integer $m < n$, each quadratic extension $E/F$ of $p$-adic local fields, and each supercuspidal representations $\pi$ of $\GL_m(E)$, corresponding to an $m$-dimensional irreducible Weil representation $\rho$, we have

\[ \gamma(s, \pi \eta, \mathscr R, \psi) = \gamma(s, \otimes \Ind_{E/F}(\rho \eta), \psi) \]

for all sufficiently highly ramified characters $\eta$ of $W_E$, with necessary degree of ramification depending on $E/F$, $\pi$, and $\psi$.
\end{IH}

Assume for the rest of this section that the induction hypothesis holds for a given $n$.  Let $E/F$ be a quadratic extension of $p$-adic fields, $\pi$ a supercuspidal representation of $\GL_n(E)$, and $\rho$ the $n$-dimensional irreducible representation of $W_E$ corresponding to $\pi$.  Our first step in showing that the stable equality theorem holds for $n+1$ is to show ``equality at a base point.''

\newtheorem{D5}[D1]{Propositon}

\begin{D5} (Equality at a base point) Let $n \geq 2$ be an integer, and assume the induction hypothesis for $n$.  Let $\omega_0$ be any character of $E^{\ast}$.  Then exists an $n$-dimensional irreducible representation $\rho_0$ of $W_E$, with $\det \rho_0$ corresponding to $\omega_0$ by local class field theory, such that if $\pi_0$ is the supercuspidal representation of $\GL_n(E)$ corresponding to $\rho_0$, and $\eta$ is any character of $E^{\ast}$, then

\[ \gamma(s, \pi(\rho_0) \eta, \mathscr R, \psi) = \gamma(s,\otimes \Ind_{E/F}(\rho_0 \eta), \psi) \]

\end{D5}

Our proof of equality at a base point will be a global argument.  We start with a lemma of Henniart.

\newtheorem{D6}[D1]{Lemma}

\begin{D6}Let $\omega_0$ be the character in Proposiiton 2.5.5.  There exists a quadratic extension of number fields $K/k$, and an $n$-dimensional complex representation $\Sigma$ of $W_K$, with the following properties:

(i): There exist places $w_0$ and $v_0$ of $K$ and $k$, respectively, such that $k_{v_0} = F$ and $K_{w_0} = E$.

(ii): The representation $\rho_0 := \Sigma_{w_0}$ of $W_{K_{w_0}} = W_E$ is irreducible, and $\det \Sigma_{w_0}$ corresponds to $\omega_0$ by local class field theory.

(iii): At all finite places $w$ of $K$ with $w \neq w_0$, the representation $\Sigma_{w_0}$ is not irreducible.

(iv): If $\pi_w$ is the representation of $\GL_n(K_{w})$ corresponding to the semisimplication $(\Sigma_w)_{\textrm{ss}}$ of $\Sigma_w$ by the local Langlands correspondence, then $\Pi = \bigotimes_w \pi_w$ is a cuspidal automorphic representation of $\GL_n(\mathbb A_K)$.

\end{D6}

\begin{proof} This is essentially Lemma 3.1 of [CoShTs].  The only difference is the lemma only mentions the field $K$, not the field $k$. In the proof, one fixes the unramified extension $M$ of $E$ of degree $n$, and produces a degree $n$ extension of number fields $\mathbb M/K$ and an extension of places $w'/w$ for which $\mathbb M_{w'} = M, K_w = E$.  We can choose $k, \mathbb M$, and $K$ at the same time to satisfy the hypotheses of our modified lemma.  \end{proof}

\begin{proof} (of Proposition 5.5.5) Following the global template (Section 2.4), we have
\[ \prod\limits_{v \in S} \gamma(s,\pi_v,\mathscr R,\Psi_v) = \prod\limits_{v \in S} \gamma(s, \otimes \Ind_{K_w/k_v} \Sigma_w,\Psi_v) \]
for a finite set $S$ of finite places of $k$ containing $v_0$, where each place $v$ of $S$ has only one place $w$ of $K$ lying over it.  Let $T = S - \{v_0\}$.  For each place $v$ of $T$, we know that $\Sigma_w$ is not irreducible, so the semisimplification $(\Sigma_w)_{\textrm{ss}}$ decomposes as $\Sigma_{w,1} \oplus \cdots \Sigma_{w,r_w}$ for an irreducible representation $\Sigma_{w,i}$ of $W_{K_w}$ of dimension $n_{w,i} < n$.  Let $\Pi_{w,i}$ be the supercuspidal representation $\GL_{n_i}(K_w)$ corresponding to $\Sigma_{w,i}$.  By the induction hypothesis, we have for all sufficiently highly ramified characters $\mathscr X_w$ of $W_{K_w}$, that
\[ \gamma(s, \Pi_{w,i} \mathscr X_w, \mathscr R, \psi) = \gamma(s, \otimes \Ind_{K_w/k_v} \Sigma_{w,i} \mathscr X_w, \psi) \tag{2.5}\]
We apply Lemma 2.5.3 to find a unitary character $\mathscr X = \otimes \mathscr X_w$ of $\mathbb A_K^{\ast}/K^{\ast}$, such that $\mathscr X_w$ is unramified for all finite places $w$ with $w \mid v$ and $v \not\in S$, $\mathscr X_{w_0}$ agrees with $\omega_0$ on $\mathcal O_E^{\ast}$, and for each $v \in T$ with $w \mid v$, $\mathscr X_w$ is sufficiently highly ramified such that (1) holds for all $1 \leq i \leq r_w$.  

We now apply the global template again, replacing $\Sigma$ by $\Sigma \mathscr X$, and $\Pi$ by $\Pi \mathscr X$, so that 
\[ \prod\limits_{v \in S} \gamma(s,\Pi_w \mathscr X_w,\mathscr R,\Psi_v) = \prod\limits_{v \in S} \gamma(s, \otimes \Ind_{K_w/k_v} (\Sigma_w \mathscr X_w),\Psi_v) \]
We have the same $S$ as before, since $\mathscr X_w$ is unramified for all finite places $v \not\in S$, $w \mid v$.  For $v \in T$, and $w \mid v$, we have that $\Pi_{w,i} \mathscr X_w, 1 \leq i \leq r_w$, is the supercuspidal support of $\Pi_w \mathscr X_w$, so that after relabeling the $\Pi_{w,i} \mathscr X_w$,
\[ \Pi_w \mathscr X_w \subset \Ind^{\GL_n(E)} (\Pi_{w,1} \mathscr X_w) \boxtimes \cdots \boxtimes (\Pi_{w,r_w} \mathscr X_w) \]
and therefore multiplicativity (Proposition 2.3.1) gives 
\begin{equation*}
    \begin{split}
        \gamma(s, \Pi_w \mathscr X_w, \mathscr R, \Psi_v) = & \prod\limits_{i=1}^{r_w} \gamma(s, \Pi_{w,i} \mathscr X_w, \mathscr R, \Psi_v)  \prod\limits_{1 \leq i < j \leq r_w} \lambda(K_w/k_v,\Psi_v)^{n_{w_i}n_{w_j}}\\ &  \gamma(s, (\Pi_{w,i} \mathscr X_w) \times (\Pi_{w,j} \mathscr X_w \circ \sigma),\Psi_v \operatorname{Tr}_{w/v})
    \end{split}
\end{equation*}
where $\sigma$ is the nontrivial element of $\Gal(K_w/k_v)$.  On the other hand, $\Sigma_w \mathscr X_w$ has the irreducible representations $\Sigma_{w,i} \mathscr X_w$ as the factors of its composition series, so by Lemma 2.1.3, we have

\begin{equation*}
    \begin{split}
        \gamma(s, \otimes \Ind_{K_w/k_v} \Sigma_w \mathscr X_w,  \Psi_v) = & \prod\limits_{i=1}^{r_w} \gamma(s, \otimes \Ind_{K_w/k_v} \Sigma_{w,i} \mathscr X_w, \Psi_v) \\& \prod\limits_{1 \leq i < j \leq r_w} \gamma(s, \Ind_{K_w/k_v} \Sigma_{w_i} \mathscr X_w \otimes (\Sigma_{w_j} \mathscr X_w \circ \iota_z),\Psi_v)
    \end{split}
\end{equation*}

where $z$ is an element of $W_{k_v}$ which is not in $W_{K_w}$.  By equation (2.5) and remark at the end of Section 2.3, we have the equality of gamma factors at every place $v \in T$.  Thus we have 
\[ \gamma(s, \pi_0 \mathscr X_{w_0}, \mathscr R, \psi) = \gamma(s, \otimes \Ind_{E/F}(\rho_0 \mathscr X_{w_0}), \psi) \]
Since $\mathscr X_{w_0}$ and $\omega_0$ agree on $\mathcal O_E^{\ast}$, they differ by an unramified character, so Lemma 2.5.1 completes the proof of equality at a base point.  
\end{proof}

To proceed with the proof of stable equality, we will need the following stability results on both sides.

\newtheorem{D7}[D1]{Proposition}

\begin{D7} (Arithmetic stability) Let $\rho_1, \rho_2$ be two representations of $W_E$ with $\det \rho_1 = \det \rho_2$.  Then for all sufficiently highly ramified characters $\eta$ of $W_E$, we have
\[ \gamma(s, \otimes \Ind_{E/F} (\rho_1\eta), \psi) = \gamma(s, \otimes \Ind_{E/F} (\rho_2\eta), \psi)\]

\end{D7}

\begin{proof} It is a consequence of Deligne's proof of the existence of the local epsilon factors that if $\det \rho_1 = \det \rho_2$, then $\gamma(s, \rho_1 \eta, \Psi) = \gamma(s, \rho_2 \eta \Psi)$ for all characters $\Psi$ of $E$ and all sufficiently highly ramified characters of $\eta$ of $W_E$ [De].  We need only observe that under our hypothesis, we have $\det \otimes \Ind_{E/F} \rho_1 = \det \otimes \Ind_{E/F} \rho_2$, $\otimes \Ind_{E/F}(\rho_i \eta) = (\otimes \Ind_{E/F}\rho_i)(\otimes \Ind_{E/F} \eta)$, and that $(\otimes \Ind_{E/F} \eta)$ is highly ramified if $\eta$ is (Lemma 2.1.4 (ii)).
\end{proof}

The stability result on the analytic side, Proposition 2.5.8, is much more difficult.  Our proof, which is the content of Section 3, is purely local, and mirrors the approach taken by Shahidi, Cogdell, and Tsai in [CoShTs].  

\newtheorem{D8}[D1]{Proposition}
\begin{D8} (Analytic stability) Let $\pi_1, \pi_2$ be supercuspidal representations of $\GL_n(E)$ with the same central character.  Then for all sufficiently highly ramified characters $\eta$ of $E^{\ast}$, we have
\[ \gamma(s, \pi_1 \eta, \mathscr R,\psi) = \gamma(s, \pi_2 \eta, \mathscr R,\psi)\]
\end{D8}

\begin{proof} This proposition will be shown to be equivalent to Proposition 3.3.3, whose proof will occupy the entirety of Section 3.\end{proof}

Granting the analytic stability result, Proposition 2.5.8, we can now finally finish the proof of stable equality (Proposition 2.5.2).  Let $\pi$ be a supercuspidal representation of $\GL_n(E)$, and $\rho$ the corresponding $n$-dimensional irreducible representation of $W_E$.  Let $\omega_0$ be the central character of $\pi$, identified with a character of $W_E$, so that $\det \rho = \omega_0$.  By Proposition 2.5.5, there exists an $n$-dimensional irreducible representation $\rho_0$ of $W_E$ with $\det \rho_0 = \det \rho$, such that if $\pi_0$ is the supercuspidal representation $\GL_n(E)$ corresponding to $\rho_0$, then
\[ \gamma(s, \pi_0 \eta, \mathscr R, \psi) = \gamma(s, \otimes \Ind_{E/F}(\rho_0 \eta), \psi) \]
for all characters $\eta$ of $E^{\ast}$.  Now $\pi$ and $\pi_0$ have the same central character $\omega_0$.  Taking $\eta$ to be very highly ramified, we have by Propositions 2.5.7 and 2.5.8,
\begin{equation*}
    \begin{split}
        \gamma(s, \pi \eta, \mathscr R,\psi) & = \gamma(s, \pi_0 \eta, \mathscr R, \psi) \\
        & = \gamma(s, \otimes \Ind_{E/F}(\rho_0 \eta), \psi)\\
        & = \gamma(s, \otimes \Ind_{E/F}(\rho \eta),\psi)
    \end{split}
\end{equation*}
This completes the proof of the induction step, and the proof of Proposition 2.5.2.

\newtheorem{D9}[D1]{Corollary}

\begin{D9} (Stable equality for general representations) Let $\pi$ be a smooth, irreducible representation of $\GL_n(E)$, corresponding to a Frobenius semisimple representation $\rho$ of $W_E'$.  Then for all sufficiently highly ramified characters $\eta$ of $W_E$, we have

\[ \gamma(s,\pi \eta, \mathscr R,\psi) = \gamma(s, \otimes \Ind_{E/F}(\rho \eta),\psi)\]

\end{D9}

\begin{proof} This follows from Proposition 2.5.2 and the following facts:

\begin{enumerate}
\item If $\pi_1, ... , \pi_r$ are supercuspidal representations of smaller $\GL$s with $\pi$ isomorphic to a subrepresentation of $\Ind^{\GL_n(E)}\pi_1 \boxtimes \cdots \boxtimes \pi_r$, and $\rho_i$ is the irreducible Weil representation corresponding to $\pi_i$, then the underlying Weil representation of $\rho$ is the direct sum of the $\rho_i$.

    \item The gamma factor on the Artin side depends only on the underlying Weil representation.
    
    \item The local Langlands correspondence is compatible by twisting of characters.
    
    \item Multiplicativity of gamma factors on both sides (Lemma 2.1.2 and Proposition 2.3.1), as well as the remark below Proposition 2.3.1 which relates Rankin products.
\end{enumerate}

\end{proof}

\subsection{Equality for monomial representations}

In this section, we prove the equality of gamma factors when $\rho$ is a monomial representation, which is to say a representation of $W_E$ which is induced from a finite order character of a finite Galois extension of $E$.

\newtheorem{E1}{Proposition}[subsection]

\begin{E1} (Equality for monomial representations) Let $E \subseteq L \subseteq M \subseteq \overline{F}$ be fields with $M$ a finite Galois extension of $E$, and $n = [L : E]$.  Let $\chi$ be a character of $\Gal(M/L)$, and let $\rho = \Ind_{L/E}(\chi)$.  Let $\pi$ be the representation of $\GL_n(E)$ corresponding to $\rho$.  Then
\[ \gamma(s, \pi, \mathscr R, \psi) = \gamma(s, \otimes \Ind_{E/F} \rho, \psi) \]
\end{E1}

The representation $\rho$ need not be irreducible, but being a Galois representation, it decomposes into a direct sum of irreducible representations $\rho_1, ... , \rho_r$ with degrees $n_1, ... , n_r$.  Let $\pi_i$ be the supercuspidal representation of $\GL_{n_i}(E)$ corresponding to $\rho_i$.  Then since $\rho$ is a Weil representation (as opposed to a Weil-Deligne representation), $\pi$ is fully induced from $\pi_1, ... , \pi_r$.

We will require the following lemma of Henniart:

\newtheorem{E2}[E1]{Lemma}

\begin{E2}There exist number fields

\[ k \subset K \subset \mathbb L \subset \mathbb M \]

together with a character $\mathscr Y$ of $\mathbb A_{\mathbb L}^{\ast}/\mathbb L^{\ast}$, and a place $w_0''$ of $\mathbb M$, such that:

\begin{enumerate}
    \item If $w_0',w_0,v$ are the places of $\mathbb L, K, k$ over which $w_0''$ lies, then $F = k_{v_0}, E = K_{w_0}, L = \mathbb L_{w_0'}, M = \mathbb M_{w_0''}$.
    
    \item $\mathbb M$ is Galois over $K$, with $[\mathbb M : K] = [\mathbb M_{w_0''} : K_{w_0}] = [M : E]$.  Hence $w_{0}''$ is the only place of $\mathbb M$ lying over $w_0$.
    
    \item $\mathscr Y_{w_0'} = \chi$ under local class field theory.
    
    \item Let $\Sigma = \Ind_{\Gal(\mathbb M/\mathbb L)}^{\Gal(\mathbb M/K)}(\mathscr Y)$.  Then there is a cuspidal automorphic representation $\Pi = \otimes_w \Pi_w$ of $\GL_m(\mathbb A_K)$ (where $m = [L : E]$) such that $\Sigma_{w_0} = \pi$.
\end{enumerate}

\end{E2}

This is Lemma 3.2 of [CoShTs].  Like Lemma 2.5.6, the original statement of this lemma did not include the field $k$, but we can easily modify the construction to include it.

Following the template of the global argument (Proposition 2.4), we have as usual
\[ \prod\limits_{v \in S} \gamma(s, \Pi_w, \mathscr R, \Psi_v) = \prod\limits_{v \in S} \gamma(s, \otimes \Ind_{K_w/k_v} \Sigma_w, \Psi_v) \]
where $S$ is a finite set of finite places $v$ of $k$ containing $v_0$, each of which has only one place $w$ of $K$ lying over it.  Let $\mathscr X = \otimes \mathscr X_w$ be a character of $\mathbb A_K^{\ast}/K^{\ast}$ which is very highly ramified at $v \in S - \{v_0\}$ and unramified at all other finite places.  We replace $\Sigma$ by $\Sigma \mathscr X$ in the template of the global argument, giving us
\[ \prod\limits_{v \in S} \gamma(s, \Pi_w \mathscr X_w, \mathscr R, \Psi_v) = \prod\limits_{v \in S} \gamma(s, \otimes \Ind_{K_w/k_v} (\Sigma_w\mathscr X_w), \Psi_v) \]
for the same set $S$.  By Corollary 2.5.9, and the fact that each $\mathscr X_w$ for $v \in S - \{v_0\}$ is highly ramified, we have equality of the gamma factor at all places $v \neq v_0$, and therefore
\[ \gamma(s, \pi \mathscr X_{w_0}, \mathscr R,\psi) = \gamma(s, \otimes \Ind_{E/F}(\rho \mathscr X_{w_0}),\psi) \]
Now $\mathscr X_{w_0}$ is unramified and therefore of the form $|| \cdot ||^{s_0}$ for some $s_0 \in \mathbb C$, so Lemma 2.5.1 concludes the proof of the proposition.

\subsection{Equality for Galois representations}

Using the equality of gamma factors for monomial representations and Brauer's theorem, we will prove the equality of gamma factors for all irreducible Galois representations.

\newtheorem{F1}{Proposition}[subsection]

\begin{F1}
(Equality for Galois representations) Let $\rho$ be an irreducible $n$-dimensional Galois representation, and let $\pi$ be the corresponding supercuspidal representation of $\GL_n(E)$.  Then

\[ \gamma(s, \pi, \mathscr R, \psi) = \gamma(s, \otimes \Ind_{E/F} \rho, \psi) \]

\end{F1}

\begin{proof} Choose $z \in W_F$ which is not in $W_E$.  By Brauer's theorem, there exist monomial representations $\rho_1, ... , \rho_r, \rho_1', ... , \rho_t'$, not necessarily all distinct, such that 

\[ \Sigma := \rho \oplus \bigoplus\limits_{i=1}^t  \rho_i' = \bigoplus\limits_{i=1}^r  \rho_i \]

We can compute $\gamma(s, \otimes \Ind_{E/F} \Sigma,\psi)$ in two ways using Lemma 2.1.2.  First,

\begin{equation*}
    \begin{split}
        \gamma(s, \otimes \Ind_{E/F} \Sigma,\psi) & = \prod\limits_{i=1}^r \gamma(s, \otimes \Ind_{E/F} \rho_i,\psi)  \prod\limits_{1 \leq i < j \leq r} \gamma(s, \Ind_{E/F} \rho_i \otimes (\rho_j \circ \iota_z),\psi)
    \end{split}
\end{equation*}

Second, we can write $\gamma(s, \otimes \Ind_{E/F} \Sigma,\psi)$ as 
\begin{equation*}
\begin{split}
\gamma(s, \otimes \Ind_{E/F} \Sigma,\psi)  = & \gamma(s, \otimes \Ind_{E/F} \rho,\psi) \prod\limits_{i=1}^t \gamma(s, \otimes \Ind_{E/F} \rho_i',\psi) \\ & \prod\limits_{1 \leq j \leq t} \gamma(s, \Ind_{E/F} \rho \otimes (\rho_j' \circ \iota_z),\psi) \prod\limits_{1 \leq i < j \leq t} \gamma(s, \Ind_{E/F} \rho_i' \otimes (\rho_j' \circ \iota_z),\psi) 
\end{split}
\end{equation*}

Let $\pi_1, ... , \pi_r, \pi_1', ... , \pi_t'$ be the smooth, irreducible representations corresponding to the Galois representations $\rho_1, ... , \rho_r, \rho_1', ... , \rho_t'$.  Let $n_i, n_i'$ be the degrees of $\pi_i, \pi_i'$.  Let $\Pi$ be the smooth, irreducible representation of $\GL_n(E)$ corresponding to $\Sigma$.  On the analytic side, we have

\[ \Pi = \Ind^{\GL_n(E)} \pi \boxtimes \pi_1 \boxtimes \cdots \boxtimes  \pi_r = \Ind^{\GL_n(E)} \pi_1' \boxtimes \cdots \boxtimes \pi_t'\]

Then we can apply multiplicativity of gamma factors (Proposition 2.3.1) in two different ways:

\begin{equation*}
    \begin{split}
        \gamma(s,\Pi,\mathscr R,\psi) =  \prod\limits_{i=1}^t \gamma(s, \pi_i', \mathscr R, \psi) \\
        \prod\limits_{1 \leq i < j \leq t} \lambda(E/F,\psi)^{n_i'n_j'}\gamma(s, \pi_i' \times (\pi_j' \circ \sigma),\psi \circ \operatorname{Tr}_{E/F})
    \end{split}
\end{equation*} 

and 

\begin{equation*}
    \begin{split}
        \gamma(s,\Pi,\mathscr R,\psi) =& \gamma(s,\pi,\mathscr R,\psi) \prod\limits_{i=1}^r \gamma(s, \pi_i, \mathscr R, \psi) \\&
        \prod\limits_{j=1}^r \lambda(E/F, \psi)^{nn_i}\gamma(s, \pi \times (\pi_j \circ \sigma), \psi \circ \operatorname{Tr}_{E/F})\\&\prod\limits_{1 \leq i < j \leq r} \lambda(E/F,\psi)^{n_in_j}\gamma(s, \pi_i \times (\pi_j \circ \sigma),\psi \circ \operatorname{Tr}_{E/F})
    \end{split}
\end{equation*}

Considering the first way we have applied multiplicativity on both sides, we have that \[\lambda(E/F,\psi)^{n_i'n_j'}\gamma(s, \pi_i' \times (\pi_j' \circ \sigma),\psi \circ \operatorname{Tr}_{E/F}) = \gamma(s,\Ind_{E/F} (\rho_i \otimes (\rho_j \circ \iota_z)), \psi)\] and we also have the equality of $\gamma(s, \pi_i', \mathscr R, \psi)$ and $\gamma(s, \otimes \Ind_{E/F} \rho_i',\psi)$ by Proposition 2.6.1.  Thus

\[ \gamma(s, \Pi, \mathscr R,\psi) = \gamma(s, \otimes \Ind_{E/F} \Sigma, \psi) \]

Now we write out the left and right hand sides of this last equation in the second way that we have described each.  We match up the terms again using the same arguments, giving us

\[ \gamma(s,\pi,\mathscr R,\psi) = \gamma(s, \otimes \Ind_{E/F} \rho,\psi) \]

This completes the proof of equality for Galois representations.  \end{proof}

We now complete the proof of Theorem 2$'$.  Let $\pi$ be a supercuspidal representation of $\GL_n(E)$, and $\rho$ the corresponding $n$-dimensional irreducible representation of $W_E$.  Then there exists an unramified character $\eta$ of $W_E$ such that $\rho \eta$ is a Galois representation (2.2.1 of [Ta]).  Identifying $\eta$ as a character of $\GL_n(E)$ through the determinant, we know that $\pi \eta$ corresponds to $\rho \eta$ under the local Langlands correspondence.  Proposition 2.7.1 tells us that Theorem 2$'$ holds for $\pi \eta$ and $\rho \eta$.  Since the gamma factors on both sides are compatible by twisting by unramified characters (Lemma 2.5.1), we see that Theorem 2$'$ holds for supercuspidal representations.

Having established Theorem 2$'$ for supercuspidals, we have Theorem 2$'$ for general representations by the same reductions as in Corollary 2.5.9.  

\section{Analytic Stability}

In this section, we prove the analytic stability result, Proposition 2.5.8.  We begin by explaining how the Asai gamma factor $\gamma(s,\pi, \mathscr R,\psi)$ arises from the Langlands-Shahidi method, by embedding $\Res_{E/F} \GL_n$ as a maximal Levi subgroup $\mathbf M$ of the even unitary group $\mathbf G = U(n,n)$.  

The Asai gamma factor is equal to the Shahidi local coefficient $C_{\chi}(s,\pi)$ (Section 3.3), up to a root of unity.  Proposition 2.5.8 is then equivalent to the stability of this local coefficient (Proposition 3.3.3).  In [Sh1], Shahidi shows how his local coefficient can be expressed as a Mellin transform of a partial Bessel function, under certain conditions.  Our approach is to apply Shahidi's local coefficient formula, and then analyze these partial Bessel functions in a similar manner to [CoShTs].  

If $\alpha$ is the simple root corresponding to $\mathbf M$, one of Shahidi's assumptions (Assumption 5.1 of [Sh1]) for his local coefficient formula is the existence of an injection $\alpha^{\vee}: F^{\ast} \rightarrow Z_{\mathbf M}(F)/Z_{\mathbf G}(F)$ such that $\alpha \circ \alpha^{\vee} = 1$.  Unfortunately, this assumption is false in our case.  But this difficulty is not too serious: we can embed $\mathbf G$ in a larger group $\widetilde{\mathbf G}$, having the same derived group as $\mathbf G$, for which the assumption holds.  Since local coefficients only depend on the derived group (Section 3.5), we will be able to apply Shahidi's formula after all.

\subsection{Definition of the Unitary Group}

Let $W$ be the $2n$ by $2n$ matrix

\[ W = \begin{pmatrix} & I_n \\ -I_n \end{pmatrix} \]

where $I_n$ is the $n$ by $n$ identity matrix.  The unitary group $\mathbf G = U(n,n)$ is defined to be an outer form of $\GL_{2n}$ with the following Galois action for $X \in \GL_{2n}(\overline{F})$ and $\gamma \in \Gal(\overline{F}/F)$:

\[ \gamma.X = \begin{cases} \gamma(X) & \textrm{ if $\gamma|_E = 1_E$} \\ W \, ^t \gamma(X)^{-1}W^{-1} & \textrm{ if $\gamma|_E \neq 1_E$} \end{cases} \]

 where $\gamma(X)$ denotes the entrywise action of $\gamma$ on $X$.  In particular, we have 

\[ \mathbf G(E) = \GL_{2n}(E) \]

and in fact $\mathbf G$ splits over $E$, with $\mathbf G \times_F E = \GL_{2n,E}$.  Moreover, we see that $\Gamma = \Gal(E/F)$ acts on $\GL_{2n}(E)$ by

\[ \sigma.X = W^t \overline{X}^{-1}W^{-1} \]

where $\overline{X}$ is the entrywise application of the nontrivial element $\sigma$ of $\Gal(E/F)$ to $X$, and so

 \[ \mathbf G(F) = \{ X \in \GL_{2n}(E) : W \, ^t \overline{X}^{-1}W^{-1} = X \} \]

If we start with $\operatorname{SL}_{2n}$ instead of $\GL_{2n}$, we can define $\operatorname{SU}(n,n)$, the \textbf{special unitary group} in the same way, and in fact we have $\operatorname{SU}(n,n) = \mathbf G_{\der}$, the derived group of $\mathbf G$.  

The verification of the following details are straightforward, and we omit the proofs.

\newtheorem{1_1}{Proposition/Definition}[subsection]

\begin{1_1} (i): Let $\mathbf T$ be the usual maximal torus of $\GL_{2n}$, and let $\mathbf S$ be the subtorus of $\mathbf T$ defined by

\[ \mathbf S(\overline{F}) = \{ \begin{pmatrix} x_1 \\ & \ddots \\ & & x_n \\ & & & x_1^{-1} \\ & & & & \ddots \\ & & & & & x_n^{-1} \end{pmatrix} \} \]

Then $\mathbf S$ is the maximal $F$-split subtorus of $\mathbf T$, and in fact $Z_{\mathbf G_{\der}}(\mathbf S) = \mathbf T$.  Hence $\mathbf G$ is quasi-split over $F$, and $\mathbf S$ is a maximal $F$-split torus of both $\mathbf G_{\der}$ and $\mathbf G$.  

(ii): Let $\epsilon_1, ... , \epsilon_n$ be the basis of $X(\mathbf S)$ such that $\epsilon_i$ sends the above matrix to $x_i$.  Then

\[ \Delta_F = \{ \epsilon_1 - \epsilon_2, ... , \epsilon_{n-1} - \epsilon_n, 2 \epsilon_n \} \]

is a set of simple roots for $\mathbf S$ in $\mathbf G$.  The corresponding relative root system is of type $C_n$.  

(iii): Let $\mathbf B$ be the Borel subgroup (minimal parabolic $F$-subgroup) of $\mathbf G$ corresponding to $\Delta_F$.  Let $e_1, ... , e_{2n}$ be the usual basis for $X(\mathbf T)$.  The set $\Delta$ of simple roots of $\mathbf T$ in $\mathbf G$ corresponding to $\mathbf B$ is $A \cup B \cup A'$, where

\[A =  \{e_1 - e_2, ... , e_{n-1} - e_n\} \]

\[ A' = \{ -(e_{n+1} - e_{n+2}), ... , -(e_{2n-1} - e_{2n}) \} \]

\[ B = \{ e_n - e_{2n} \} \]

(iv): The nontrivial element $\sigma \in \Gamma = \Gal(E/F)$ switches $e_i$ and $-e_{n+i}$ for $1 \leq i \leq n$.  Hence the orbits of $\Delta$ under the action of $\Gamma$ are

\[ \{e_i - e_{i+1}, -(e_{n+i} - e_{n+i+1}) \} \]

for $1 \leq i \leq n-1$, as well as the singleton set $\{e_n - e_{2n} \}$.  In particular, $e_n - e_{2n}$ is the only simple root in $\Delta$ which is defined over $F$.

(v): Let $\theta = \Delta_F - \{2 \epsilon_n \}$, and let $\mathbf P$ be the corresponding maximal $F$-parabolic subgroup of $\mathbf G$.  It is self-associate.  Let $\mathbf N$ be the unipotent radical of $\mathbf P$, and let $\mathbf M$ be the unique Levi subgroup of $\mathbf P$ containing $\mathbf T$.  Then 
\[ \mathbf M(F) = \{ \begin{pmatrix} x \\ & ^t \overline{x}^{-1} \end{pmatrix} : x \in \GL_n(E) \} \]

\[ \mathbf N(F) = \{ \begin{pmatrix} I_n & X \\ & I_n \end{pmatrix}: X \in \operatorname{Mat}_n(E), \, ^t\overline{X} = X \} \]

\end{1_1}

\subsection{The L-group of $\mathbf G$}

We can identify the L-group of $^L\mathbf G$ with the semidirect product of $\GL_{2n}(\mathbb C)$ by $\Gal(E/F)$, where $\Gal(E/F)$ acts on $^L\mathbf G$ by $\sigma.X = W^tX^{-1}W^{-1}$.  The L-group of $\mathbf M$ can be identified with the semidirect product of $\GL_n(\mathbb C) \times \GL_n(\mathbb C)$ by $\Gal(E/F)$, which acts by $\sigma.(x,y) = (^tx^{-1},^ty^{-1})$.  The Lie algebra $^L\mathfrak n$ of the L-group of $\mathbf N$ identifies with $\operatorname{Mat}_n(\mathbb C)$, and the adjoint representation $r: \, ^L\mathbf M \rightarrow \GL(^L\mathfrak n)$ is given by
\[ r(x,y,1).X = xXy^{-1}\]
\[ r(\sigma).X = ^tX \]
It is irreducible.  
\newtheorem{1_2}[1_1]{Lemma}

\begin{1_2} Let $\pi$ be an irreducible, admissible representation of $\mathbf M(F) = \GL_n(E)$.  Then
\[ \gamma(s,\pi,r,\psi)=\gamma(s,\pi,\mathscr R,\psi)\]
where $\mathscr R$ is the Asai representation.  
\end{1_2}
\begin{proof} We can take $\Res_{E/F} \GL_n$ to be the group defined on $E$-points by $\GL_n(E) \times \GL_n(E)$, with $\Gal(E/F)$ acting by switching the factors.  Define an isomorphism $\Res_{E/F} \GL_n \rightarrow \mathbf M$ of algebraic groups over $F$ by
\[ (x,y) \mapsto \begin{pmatrix} x \\ & ^ty^{-1}\end{pmatrix} \]
Let $V$ be an $n$-dimensional complex vector space with basis $e_1, ... e_n$.  We can identify the space $^L\mathfrak n$ with $V \otimes V$, an isomorphism being given by sending the elementary matrix $E_{ij}$ to $e_i \otimes e_j$.  By identifying $\mathbf M$ with $\Res_{E/F} \GL_n$ via the isomorphism given, and taking the corresponding isomorphism of L-groups $^L\Res_{E/F} \GL_n \rightarrow \, ^L\mathbf M$, it is straightforward to check that $r$ now identifies with the representation $\mathscr R$.\end{proof}

\subsection{The local coefficient for $\mathbf M$ inside $\mathbf G$}

Let $\alpha = 2 \epsilon_n$ be the simple root of $\Delta_F$ which defines the maximal parabolic subgroup $\mathbf P$, and let $\rho$ be half the sum of the roots of $\mathbf S$ in $\mathbf N$.  We will calculate the element $\tilde{\alpha} \in X(\mathbf M)_F$ as defined in Section 2 of [Sh2].  We see that
\[ \rho = n(\epsilon_1 + \cdots + \epsilon_n)\]
and using the usual Weyl group and Galois group invariant bilinear form on $X(\mathbf T) \otimes_{\mathbb Z} \mathbb R$ we obtain $\tilde{\alpha} = \langle \rho, \alpha \rangle^{-1} \rho = n^{-1}\rho$.  Then we have:

\newtheorem{1_3}[1_1]{Lemma}

\begin{1_3} (i) For $s \in \mathbb C$, and all $m \in \GL_n(E) = \mathbf M(F)$,
\[ q_F^{\langle s \rho, H_{\mathbf M}(m) \rangle} = |\det m|_E^{ns/2}\]
\[ q_F^{\langle s\tilde{\alpha}, H_{\mathbf M}(m) \rangle} = |\det m|_E^{s/2}\]
(ii) If $\pi$ is a generic representation of $\GL_n(E)$, and $s_0 \in \mathbb C$, then
\[ \gamma(s,\pi |\det(-)|^{s_0}_E, \mathscr R,\psi) = \gamma(s+2s_0,\pi, \mathscr R,\psi)\]
\end{1_3}
\begin{proof} (i) is proved in Section 2 of [Go], and (ii) follows from (i), Lemma 3.2.1, and [Sh2], Theorem 8.3.2, condition 2.  \end{proof}

We have defined a set of simple nonrestricted roots $\Delta$.  Note that this is not the usual set of simple roots, so the unipotent radical $\mathbf U$ of $\mathbf B$ is not the group of upper triangular unipotent matrices.  For each $\beta \in \Delta$, we now define root vectors $x_{\beta}: \mathbf G_a \rightarrow \mathbf U_{\beta}$ on $\overline{F}$-points.  If $\beta = e_i - e_{i+1}$, for $i = 1, ... , n-1$, we define

\[ \mathbf x_{\beta}(t) = I_{2n} + tE_{i,i+1} \]

where $E_{i,i+1}$ is the $2n$ by $2n$ matrix with a $1$ in the $(i,i+1)$ position, and zeroes elsewhere.  If $\beta = -(e_{n+i} - e_{n+i+1})$ for $i = 1, ... , n-1$, we define 

\[ \mathbf x_{\beta}(t) = I_{2n} - tE_{n+i,n+i+1} \]

Finally, if $\beta= \alpha = e_n - e_{2n}$, we define

\[ \mathbf x_{\alpha}(t) = \mathbf x_{\beta}(t) = I_{2n} + t E_{n,2n} \]

This splitting is defined over $F$.  Having defined these root vectors, we can then define our canonical Weyl group representatives as in Section 2 of [Sh1].

\newtheorem{1_4}[1_1]{Lemma}

\begin{1_4} (i): For the set of relative simple roots $\alpha_i = \epsilon_i - \epsilon_{i+1}$ for $1 \leq i \leq n -1$ and $\alpha_n = \alpha = 2 \epsilon_n$, we have from the above splitting the following canonical Weyl group representatives for the corresponding simple reflections $w_1, ... ,w_n$.  First, let

\[ C = \begin{pmatrix} 0 & 1 \\ -1 & 0 \end{pmatrix} \]

For $1 \leq i \leq n-1$, we have 

\[ \dot w_i = \begin{pmatrix} I_{i-1} \\ & C \\ & & I_{n-2} \\ & & &  C \\ &  & & & I_{n-i-1} \end{pmatrix} \]

and

\[ \dot w_n = \begin{pmatrix} I_{n-1} \\ && & 1 \\ && I_{n-1} \\ & -1 \end{pmatrix} \]

(ii): Let $J$ be the $n$ by $n$ antidiagonal matrix
\[ J = \begin{pmatrix} & & & 1 \\ & & -1 \\ & \adots \\ (-1)^{n-1} \end{pmatrix} \]
let $w_{\ell}$ and $w_{\ell}^{\theta}$ be the long elements of $\mathbf G$ and $\mathbf M$, and let $w_0 = w_{\ell} w_{\ell}^{\theta}$.  These have canonical representatives
\[ \dot w_{\ell} = (-1)^{n-1}\begin{pmatrix} & I_n \\  -I_n \end{pmatrix}, \dot w_{\ell}^{\theta} = \begin{pmatrix} J \\ & J \end{pmatrix},\dot{w_0} = \begin{pmatrix} & J \\ -J \end{pmatrix} \]
\end{1_4}
The root vectors also define a canonical generic character $\chi$ of $\mathbf U(F)$ in terms of a fixed additive character $\psi$ of $F$: if $u \in \mathbf U(F)$, we can write
\[ u = \prod\limits_{\beta \in \Delta} x_{\beta}(a_{\beta}) u'\]
for $a_{\beta} \in \overline{F}$ and $u'$ in the derived group of $\mathbf U(F)$.  The sum of the $a_{\beta} : \beta \in \Delta$ lies in $F$, and we define $\chi(u) = \psi(\sum\limits_{\beta \in \Delta} a_{\beta})$.  

If $\pi$ is a generic representation of $\GL_n(E) = \mathbf M(F)$, then the Langlands-Shahidi method defines the Shahidi local coefficient $C_{\chi}(s,\pi)$ [Sh1].  The local coefficient is related to the Asai gamma factor by the formula
\[ C_{\chi}(s,\pi) = \lambda(E/F,\psi)^{n^2} \gamma(s, \pi^{\vee}, \mathscr R, \overline{\psi})\]
where $\lambda(E/F,\psi)$ is the Langlands lambda function (Theorem 8.3.2, [Sh2]).  Hence analytic stability (Proposition 2.5.8) is equivalent to:
\newtheorem{1_5}[1_1]{Proposition}
\begin{1_5} Let $\pi_1$ and $\pi_2$ be supercuspidal representations of $\GL_n(E)$ with the same central character $\omega$.  Then for all sufficiently highly ramified characters $\eta$ of $E^{\ast}$, we have
\[ C_{\chi}(s,\pi_1 \eta) = C_{\chi}(s,\pi_2 \eta)\]
where $\pi_i \eta = \pi_i(\eta \circ \operatorname{det})$.
\end{1_5}
Our approach to Proposition 3.3.3 will be to apply Shahidi's local coefficient formula (Theorem 6.2 of [Sh1]) and then the Bessel function asymptotics of [CoShTs].  However, the group $\mathbf G$ is insufficient to apply Shahidi's formula.  In Section 3.6, we will embed $\mathbf G$ in a larger group $\widetilde{\mathbf G}$ which has the same derived group, and which has connected and cohomologically trivial center.  The group $\widetilde{\mathbf G}$ satisfies the necessary properties to apply Shahidi's formula.  As we explain in Section 3.5, local coefficients only depend on the derived group, so we will be able to calculate $C_{\chi}(s,\pi)$ using $\widetilde{\mathbf G}$. 

Let $\overline{\mathbf N}$ be the unipotent radical of the parabolic opposite to $\mathbf N$. For $n \in \mathbf N(F)$, we will need an explicit decomposition of $\dot{w}_0^{-1}n \in \mathbf P(F) \overline{\mathbf N}(F)$ as in Section 4 of [Sh1].

\newtheorem{1_6}[1_1]{Lemma}

\begin{1_6} Let 
\[  n = \begin{pmatrix} I_n & X \\ & I_n \end{pmatrix} \in \mathbf N(F) \]

for $X \in \textrm{Mat}_n(E)$ and $^t\overline{X} = X$.  Then $\dot w_0^{-1}n \in \mathbf P(F)\overline{\mathbf N}(F)$ if and only if $X$
is invertible, in which case we can uniquely express $\dot w_0^{-1}n = mn'\bar{n}$, with $m \in \mathbf M(F), n' \in \mathbf N(F), \bar{n} \in \overline{\mathbf N}(F)$.  We have

\[ m = (-1)^{n-1} \begin{pmatrix} JX^{-1} \\ & JX \end{pmatrix} \, \, n' =\begin{pmatrix} I_n  & -X \\ & I_n \end{pmatrix} \, \, \bar n =  \begin{pmatrix} I_n \\ X^{-1} & I_n \end{pmatrix} \]

\end{1_6}

\begin{proof} This is essentially Lemma 2.2 of [Go].
\end{proof}

\subsection{Orbit space measures}
Shahidi's local coefficient formula expresses $C_{\chi}(s,\pi)^{-1}$ as an integral with respect to a measure on the quotient space of $\mathbf N(F)$ with respect to the action of a certain group.  In this section, we will explicitly construct the measure which we will need, and show it satisfies the required properties.  

Let $\mathbf U_{\mathbf M} = \mathbf M \cap \mathbf U$.  The group $\mathbf U_{\mathbf M}(F)$, which identifies with the upper triangular unipotent matrices of $\GL_n(E)$, acts by conjugation on $\mathbf N(F)$, which identifies with the space of $n$ by $n$ Hermitian matrices in $\operatorname{Mat}_n(E)$.  Under these identifications, the action of $\mathbf U_{\mathbf M}(F)$ on $\mathbf N(F)$ is given by
\[ u.X = uX^t\overline{u}\]

\newtheorem{2_1}{Proposition}[subsection]

\begin{2_1} Let $R$ be the set of elements in $\mathbf N(F)$ of the form

\[ \begin{pmatrix} I_n & r \\ & I_n \end{pmatrix} \]

where $r = \textrm{diag}(r_1, ... , r_n)$ is an invertible diagonal matrix with entries necessarily in $F$.  Then:

(i): The elements of $R$ lie in distinct orbits under the action of $\mathbf U_{\mathbf M}(F)$, each with trivial stabilizer.

(ii): The disjoint union $W$ of the orbits $\mathbf U_{\mathbf M}(F).r$ for $r \in R$ is an open dense subset of $\mathbf N(F)$.  

(iii): The map $\mathbf U_{\mathbf M}(F) \times R \rightarrow W$ given by $(u,r) \mapsto u.r$ is an isomorphism of analytic manifolds.  In particular, the map $W \rightarrow R$ sending $n$ to the unique element of $R$ lying in the same orbit is a submersion of manifolds, so $R$ is the quotient of $W$ under the action of $\mathbf U_{\mathbf M}(F)$ in the category of analytic manifolds.

(iv): Identifying $R$ with $(F^{\ast})^n$, we place the measure $dr = \prod\limits_{i=1}^n |r_i|_E^{i-1} \, dr_i$ on $R$.  Then integration over $\mathbf N(F)$ can be recovered by integration over $D$:

\[ \int\limits_{\mathbf N(F)} f(n) \, dn = \int\limits_{R} \int\limits_{\mathbf U_{\mathbf M}(F)} f(u.r) \, du \, dr \tag{$f \in \mathscr C_c^{\infty}(\mathbf N(F))$}\]

\end{2_1}

\begin{proof} Assume first that $n = 2$.  Then (i), (ii), and (iii) are straightforward to verify.  For (iv), we identify $\mathbf U_{\mathbf M}(F)$ with $E$, and $\mathbf N(F)$ with $F \times E \times F$.  Then we have for $r = \textrm{diag}(r_1,r_2) \in D$ and $f \in \mathscr C_c^{\infty}(\mathbf N(F))$,
\begin{equation*}
\begin{split}
\int\limits_{R} \int\limits_{\mathbf U_{\mathbf M}(F)} f(u.r) \, du \, dr & = \int\limits_F \int\limits_{F^{\ast}} \int\limits_E f(r_1 + x\overline{x},xr_2,r_2)|r_2|_E \, dx \, dr_2 \, dr_1 \\
& = \int\limits_{F^{\ast}} \int\limits_E \int\limits_F f(r_1 + x\overline{x},xr_2,r_2)|r_2|_E \, dr_1 \, dx \, dr_2 \\
& =  \int\limits_{F^{\ast}} \int\limits_E \int\limits_F f(r_1,xr_2,r_2)|r_2|_E \, dr_1 \, dx \, dr_2 \\
& = \int\limits_{F^{\ast}} \int\limits_E \int\limits_F f(r_1,x,r_2)\, dr_1 \, dx \, dr_2 \\
& = \int\limits_{\mathbf N(F)} f(n) \, dn
\end{split}
\end{equation*} 

From the second to the third line, we have used the translation invariance of the measure $dr_1$ on $F$.  For the third to the fourth line, we have used the fact that $\int\limits_E F(rx) dx = |r|_E^{-1} \int\limits_E F(x) dx$ for any Haar measure $dx$ on $E$.  Finally we use the fact that integration over $F^{\ast}$ is the same as integration over $F$.

We then proceed by induction on $n$.  Suppose we have verified that $R$ and $W$ work for a given $n$.  We want to prove the proposition for the corresponding sets $R^{\ast}$ and $W^{\ast}$ of size $n+1$.  Consider the dense open set $O$ of those matrices $X \in \mathbf N(F)$ (of size $n+1$) whose lower right entry $x$ is nonzero.  Write

\[ X = \begin{pmatrix} X_0 & \alpha \\ ^t \overline{\alpha} & x \end{pmatrix} \]

for $X_0$ Hermitian of size $n$ and $\alpha$ a column vector.  Let

\[ u = \begin{pmatrix} I_{n} & -x^{-1} \alpha \\ & 1 \end{pmatrix} \]

so that 

\[ un^t \overline{u} = \begin{pmatrix} X_0 - x^{-1} \alpha ^t\overline{\alpha} & 0 \\ 0 & x \end{pmatrix} \]

This procedure allows us to descend to size $n$ and utilize our induction hypothesis.  One checks that $W^{\ast}$ consists of those $X$ for which $X_0 - x^{-1} \alpha ^t\overline{\alpha}$ lies in $W$.  The map $X \mapsto X_0 - x^{-1} \alpha ^t\overline{\alpha}$ on $O$ is easily seen to be a submersion, making it in particular an open map.  Hence $W^{\ast}$ is dense by induction, giving us (ii).  The other properties (i), (iii), and (iv) are also proved by induction using this method of descent.  
\end{proof}

We also have an action of $F^{\ast}$ on $\mathbf N(F)$ by scaling each entry.  This actions commutes with that of $\mathbf U_{\mathbf M}(F)$, so we have an action of $F^{\ast} \times \mathbf U_{\mathbf M}(F)$ on $\mathbf N(F)$.  Let $R'$ be the set of invertible diagonal matrices of the form $\textrm{diag} (1, r_2, ... , r_n)$ in $\mathbf N(F)$.  Define a measure on $F^{\ast}$ by $|z|_F^{n^2} d^{\ast} z = |z|_F^{n^2-1} dz$ and a measure $dr'$ on $R' = (F^{\ast})^{n-1}$ by

\[ dr' = \prod\limits_{i=2}^n |r_i|_E^{i-1} dr_i = \prod\limits_{i=2}^n |r_i'|_F^{2i-1} d^{\ast}r_i'\]

Then we see immediately that $R'$ is the quotient of $R$ under the action of $F^{\ast}$, and that integration over $R$ can be recovered by integration over $R'$ and $F^{\ast}$:

\[ \int\limits_{R} f(r) \, dr = \int\limits_{R'} \int\limits_{F^{\ast}} f(z.r')|z|_F^{n^2-1} \, dz \, dr' \]
Putting this together with Proposition 3.4.1, we have:

\newtheorem{2_2}[2_1]{Proposition}

\begin{2_2} $R'$ is the quotient of an open dense subset of $\mathbf N(F)$ under the action of $F^{\ast} \times \mathbf U_{\mathbf M}(F)$, and for $f \in \mathscr C_c^{\infty}(F)$, integration over $\mathbf N(F)$ can be recovered by integration over $R'$ and $F^{\ast} \times \mathbf U_{\mathbf M}(F)$:

\[ \int\limits_{\mathbf N(F)} f(n) \, dn = \int\limits_{R'} \int\limits_{F^{\ast}} \int\limits_{\mathbf U_{\mathbf M}(F)} f(u.(zr')) |z|_F^{n^2} \, du \, d^{\ast}z \, dr'\]

\end{2_2}

\subsection{Reductive groups sharing the same derived group}

Consider a connected, reductive group $\widetilde{\mathbf G}$ over $F$ which contains $\mathbf G$ and shares its derived group.  In this section only, we will denote the group of rational points of a group $\mathbf H$ by the corresponding letter $H$.  We will not do this in general, because later on we will need to consider the group $U$ of upper triangular unipotent matrices in $\GL_n(E)$, and we do not want to confuse this with the group $\mathbf U(F)$ introduced earlier.

\newtheorem{3_1}{Lemma}[subsection]

\begin{3_1}Let $Z_{\widetilde{G}}$ be the center of $\widetilde{G}$.  Let $\pi$ be an irreducible, admissible representation of $G$, and $\omega$ a character of $Z_G$.

(i): $\omega$ can be extended to a character $\tilde{\omega}$ of $Z_{\widetilde{G}}$.  

(ii): If $\pi$ is an irreducible, admissible representation of $G$, then there exists an irreducible representation $\tilde{\pi}$ of $\widetilde{G}$ whose restriction to $G$ contains $\pi$ as a subrepresentation.

(iii): If $\pi$ has central character $\omega$, then $\tilde{\pi}$ can be chosen to have central character $\tilde{\omega}$.
\end{3_1}

\begin{proof} (i): The groups $Z_G$ and $Z_{\widetilde{G}}$ each have unique maximal compact open subgroups $K$ and $\tilde{K}$, with $K = Z_G \cap \tilde{K}$.  The restriction of $\omega$ to $K$ is unitary, and then extends by Pontryagin duality to a character $\tilde{\omega}$ of $\tilde{K}$.  We then extend $\tilde{\omega}$ to a character of $Z_G \tilde{K}$ by setting $\tilde{\omega}(zk) = \tilde{\omega}(z) \omega(k)$.  This is well defined, and moreover continuous since the product map $Z_G \times \tilde{K} \rightarrow Z_{\widetilde{G}}$ is an open map.  Since $\mathbb C^{\ast}$ is injective in the category of abelian groups, $\tilde{\omega}$ extends to an abstract homomorphism of $Z_{\widetilde{G}}$ into $\mathbb C^{\ast}$.  This extension is automatically continuous, because its restriction to the open subgroup $\tilde{K}$ is continuous.

(ii) and (iii): We first extend $\pi$ to a representation of $Z_{\widetilde{G}}G$ by setting $\pi(zg) = \tilde{\omega}(z)\pi(g)$.  This is smooth and admissible.  Since $Z_{\widetilde{G}}G$ is of finite index in $\widetilde{G}$, the smoothly induced representation $\sigma = \Ind_{Z_{\widetilde{G}}G}^{\widetilde{G}} \pi$ is admissible.  

Any irreducible subrepresentation of $\sigma$ is easily seen to have central character $\tilde{\omega}$.  Take a nonzero element in the space of $\sigma$ and consider the $\widetilde{G}$-subrepresentation $W$ which it generates.  Since $W$ is finitely generated and admissible, it is of finite length, and must contain an irreducible subrepresentation $W_0$.  

Now the restriction of $W_0$ to $G$ is a finite direct sum of irreducible representations of $G$ ([Tad], Lemma 2.1).  Since the map $f \mapsto f(1)$ defines a nonzero intertwining operator from $W_0$ to the space of $\pi$, we see that one of these irreducible representations must be isomorphic to $\pi$.
\end{proof}

The group $\widetilde{\mathbf G}$ is essentially the same as the group $\mathbf G$, but with a larger maximal torus $\widetilde{\mathbf T}$, which we may take to be one containing $\mathbf T$.  It has a maximal split torus $\widetilde{\mathbf S}$ inside $\widetilde{\mathbf T}$.  It has a Borel subgroup $\widetilde{\mathbf B} = \widetilde{\mathbf T} \mathbf U$, which defines a set of simple restricted and nonrestricted roots identifiable with those corresponding to the triple $\mathbf G, \mathbf B, \mathbf S$.  The root vectors can be taken from $\mathbf U$, giving us the exact same canonical Weyl group representatives and generic character $\chi$ as before.  

Let $\widetilde{\mathbf P} = \widetilde{\mathbf M} \mathbf N$ be the maximal self-associate parabolic subgroup of $\widetilde{\mathbf G}$ corresponding to $\mathbf P$ (that is, defined by the same set of simple roots), with $\widetilde{\mathbf M}$ containing $\widetilde{\mathbf T}$.  Then $\widetilde{\mathbf M}$ and $\mathbf M$ also have the same derived group, so we can apply Lemma 3.5.1.  We will make the following convention: for each character $\omega$ of $Z_M$, \emph{choose once and for all} a character $\tilde{\omega}$ of $Z_{\widetilde{M}}$ which extends $\omega$.

\newtheorem{3_2}[3_1]{Lemma}

\begin{3_2} Let $\pi$ be a supercuspidal representation of $M = \GL_n(E)$ with central character $\omega_{\pi}$.  Let $\tilde{\pi}$ be a representation of $\widetilde{M}$ whose restriction to $M$ contains $\pi$ as a subrepresentation, and whose central character is the given extension $\tilde{\omega}_{\pi}$ of $\omega_{\pi}$.

(i): $\tilde{\pi}$ is generic, and $C_{\chi}(s,\pi) = C_{\chi}(s, \tilde{\pi})$.

(ii): Let $W$ be an element of the Whittaker model of $\pi$.  Then $W$ extends to an element $\tilde{W}$ in the Whittaker model of $\tilde{\pi}$.

\end{3_2}

\begin{proof} (i): Let $\lambda$ be a nonzero $\chi$-Whittaker functional for $\pi$.  Let $V$ be the underlying space of $\tilde{\pi}$.  The restriction of $\tilde{\pi}$ to $M$ is a finite direct sum of irreducible representations, say $V = V_1 \oplus \cdots \oplus V_r$, with  $\pi = V_1$.  Then the map $\tilde{\lambda}: (v_1, ... , v_r) \mapsto \lambda(v_1)$  is a nonzero Whittaker functional for $\tilde{\pi}$.

Consider the induced representations $I(s,\pi)$ and $I(s,\tilde{\pi})$ of $G$ and $\widetilde{G}$.  If $0 \neq f \in I(s,\tilde{\pi})$, then $f$ is a function from $\widetilde{G}$ to $V$, so we may write $f = (f_1, ... , f_t)$.  We see immediately that the restriction of $f_1$ to $G$ is a nonzero element of $I(s,\pi)$.  Considering the intertwining operators $A(s,\pi)$ and $A(s,\tilde{\pi})$, and the Whittaker functionals $\lambda_{\chi}(s,\pi)$ and $\lambda_{\chi}(s,\tilde{\pi})$, both defined by integration over $\mathbf N(F)$ with the same Weyl group representative $\dot{w}_0$, we see by direct computation that the local coefficient $C_{\chi}(s,\tilde{\pi})$ satisfies
\[ C_{\chi}(s,\tilde{\pi}) \lambda_{\chi}(-s,\pi) \circ A(s,\pi)f_1 = \lambda(s,\pi)f_1 \]
making it equal to $C_{\chi}(s,\pi)$.

(ii): There exists an element $v$ in the space of $\pi$ such that $W(m) = \lambda(\pi(m)v)$.  We simply define $\tilde{W}(\tilde{m}) = \tilde{\lambda}(\tilde{\pi}(\tilde{m})v)$.

\end{proof}
\subsection{The extended group $\widetilde{\mathbf G}$}
We will now construct the group $\widetilde{\mathbf G}$ of the previous section.  The goal is to construct a connected, reductive group $\widetilde{\mathbf G}$ over $F$ which contains $\mathbf G$, shares its derived group, and has connected and cohomologically trivial center.  Then Assumption 5.1 of [Sh1] will be valid for $\widetilde{\mathbf G}$, allowing us to construct the injection $\alpha^{\vee}$ of $F^{\ast}$ into $Z_{\widetilde{\mathbf M}(F)}/Z_{\widetilde{\mathbf G}(F)}$ as in Section 5 of [Sh1].  This will allow us to apply Shahidi's local coefficient formula for $\widetilde{\mathbf G}$.

First, define $\widetilde{Z_{\mathbf G}} = \operatorname{Res}_{E/F} Z_{\mathbf G}$.  We can identify the $\overline{F}$-points of this group with $\overline{F}^{\ast} \times \overline{F}^{\ast}$, and for $z = (x,y) \in \widetilde{Z_{\mathbf G}}(\overline{F})$, and $\gamma \in \Gal(\overline{F}/F)$, we have $\gamma.z = (\gamma(x),\gamma(y))$ if $\gamma|_E = 1_E$, and $\gamma.z = (\gamma(y), \gamma(x))$ if $\gamma|_E = \sigma$.  

We embed $Z_{\mathbf G}$ into $\widetilde{Z_{\mathbf G}}$ on $\overline{F}$-points by sending $x \in \overline{F}^{\ast}$ to $(x,x^{-1})$, where we identify $xI_{2n}$ with $x$.  This embedding is defined over $F$.  Let $\mathbf K$ be the finite group scheme $\mathbf G_{\der} \cap Z_{\mathbf G}$, where we have $\mathbf G_{\der} = \operatorname{SU}(n,n)$.  The product map $\mathbf G_{\der} \times_F Z_{\mathbf G} \rightarrow \mathbf G$ induces an isomorphism of algebraic groups

\[ \frac{\mathbf G_{\der} \times_F Z_{\mathbf G}}{\mathbf K} \rightarrow \mathbf G \]

which is defined over $F$.  Here we are regarding $\mathbf K$ as a subgroup scheme of $\mathbf G_{\der} \times_F Z_{\mathbf G}$ on closed points by $x \mapsto (x,x^{-1})$.    Since $\mathbf K \subset Z_{\mathbf G} \subset \widetilde{Z_{\mathbf G}}$, we may define in the same way a group $\widetilde{\mathbf G}$ by

\[ \widetilde{\mathbf G} = \frac{\mathbf G_{\der} \times_F \widetilde{Z_{\mathbf G}}}{\mathbf K} \]

This group contains $\mathbf G_{\der}, Z_{\mathbf G},$ and $\mathbf G$ as subgroup schemes, and by passing to $\overline{F}$-points, we immediately arrive at the following proposition.

\newtheorem{4_1}{Proposition}[subsection]

\begin{4_1} $\widetilde{\mathbf G}$ is a connected, reductive group over $F$.  Its derived group is $\mathbf G_{\der}$.  The center of $\widetilde{\mathbf G}$ is $\widetilde{Z_{\mathbf G}}$, and  

\[ \widetilde{\mathbf T} = \frac{\mathbf T^D \times_F \widetilde{Z_{\mathbf G}}}{\mathbf K} \]

is a maximal torus of $\widetilde{\mathbf G}$ which contains $\mathbf T$ and is defined over $F$.  Here $\mathbf T^D$ is the usual maximal torus of $\mathbf G_{\der}$.  

\end{4_1}

For the self-associate maximal parabolic subgroup $\widetilde{\mathbf P} = \widetilde{\mathbf M} \mathbf N$ analoguous to $\mathbf P$ and $\mathbf M$, we have 

\[ \widetilde{\mathbf M} = \frac{ \mathbf M^D \times_F \widetilde{Z_{\mathbf G}}}{\mathbf K} \]

where $\mathbf M^D$ is the Levi subgroup of $\mathbf G_{\textrm{der}}$ analagous to $\mathbf M$.  The group $\widetilde{\mathbf M}$ has center 

\[ Z_{\widetilde{\mathbf M}} = \frac{Z_{\mathbf M^D} \times_F \widetilde{Z_{\mathbf G}}}{\mathbf K} \]
Note that the torus $Z_{\widetilde{\mathbf G}} = \widetilde{Z_{\mathbf G}}$ is cohomologically trivial by Shapiro's lemma, so the inclusion \[    Z_{\widetilde{\mathbf M}(F)}/Z_{\widetilde{\mathbf G}(F)} =Z_{\widetilde{\mathbf M}}(F)/Z_{\widetilde{\mathbf G}}(F)\subseteq Z_{\widetilde{\mathbf M}}/Z_{\widetilde{\mathbf G}}(F)\]

is an equality.  Note that if $\mathbf H$ is a reductive group over a field $k$, then $Z_{\mathbf H}(k) = Z_{\mathbf H(k)}$.  We will require a simple lemma on tori.  We omit the proof.

\newtheorem{4_3}[4_1]{Lemma}

\begin{4_3} Identify all groups with their closed points.  Let $\mathbf H$ be the subtorus $(x,x^{-1},x^{-1},x)$ of  $\mathbf G_m^4$, and $\mathbf K$ a finite subgroup of $\mathbf H$ containing $c = (-1,-1,-1,-1)$.  Choose for each $0 \neq x \in \overline{F}$ a square root $\sqrt{x}$, so that 

\[ x \mapsto (\sqrt{x}, \frac{1}{\sqrt{x}}, \frac{1}{\sqrt{x}}, \sqrt{x}) \mathbf K \]

is a well defined homomorphism of abstract groups $\mathbf G_m \rightarrow \mathbf H/\mathbf K$, independent of the choice of $\sqrt{x}$ for any $x$.  Then this homomorphism is a morphism of varieties.  
\end{4_3}

Now, we are going to construct the required injection $\alpha^{\vee}$ of $F^{\ast}$ into $Z_{\widetilde{\mathbf M}}(F)/Z_{\widetilde{\mathbf G}}(F)$ as in Section 5 of [Sh1].  Let $\mathbf L = Z_{\mathbf M^D} \times_F \widetilde{Z_{\mathbf G}}$.  Since

\[ Z_{\mathbf M^D}(\overline{F}) = \{ \begin{pmatrix} xI_n \\ & x^{-1}I_n \end{pmatrix} : x \in \overline{F}^{\ast} \} \]

we can identify $\mathbf L(\overline{F})$ with the three dimensional torus $(x,x^{-1},y,z)$.  For the corresponding group $\mathbf K$, we then identify $\mathbf K(\overline{F}) = \{ (x,x^{-1},x^{-1},x): x^{2n} = 1 \}$.  Then

\[ Z_{\widetilde{\mathbf M}}(\overline{F}) = \mathbf L/\mathbf K(\overline{F}) = \frac{\mathbf L(\overline{F})}{\mathbf K(\overline{F})} = \{ (x,x^{-1},y,z) \mathbf K(\overline{F}) : x, y, z \in \overline{F}^{\ast} \} \]

\newtheorem{4_4}[4_1]{Proposition}

\begin{4_4} For each $x \in \overline{F}^{\ast}$, choose once and for all a square root $\sqrt{x}$.  Define a map $\mathbf G_m(\overline{F}) \rightarrow Z_{\widetilde{\mathbf M}}(\overline{F})$ by 

\[ x \mapsto (\sqrt{x}, \frac{1}{\sqrt{x}}, \frac{1}{\sqrt{x}}, \sqrt{x})\mathbf K(\overline{F})\]

Then this is the map on closed points defined by a cocharacter $\lambda$ of $Z_{\mathbf M}$.  It satisfies $\langle \beta, \lambda \rangle = 1$ for the unique $\beta \in \widetilde{\Delta}$ restricting to $\alpha = \epsilon_{n-1} - \epsilon_n$, and $\langle \beta, \lambda \rangle = 0$ for $\beta \in \Delta$ not restricting to $\alpha$.  The composition 

\[ \mathbf G_m \rightarrow Z_{\widetilde{\mathbf M}} \rightarrow Z_{\widetilde{\mathbf M}}/Z_{\widetilde{\mathbf G}} \]

maps $F$ rational points to $F$-rational, and therefore defines an injection 

\[\alpha^{\vee}: F^{\ast} \rightarrow Z_{\widetilde{\mathbf M}}/Z_{\widetilde{\mathbf G}}(F)\]

\end{4_4}

\begin{proof} Note that $(-1,-1,-1,-1) \in \mathbf K(\overline{F})$, so by the previous lemma, $\lambda$ is a well defined cocharacter.  It clearly pairs with the nonrestricted simple roots in the manner described.  We finally have to check that if $x \in F^{\ast}$, then the image of $\lambda(x)$ in $Z_{\mathbf G}(\overline{F})$ is an $F$-rational point.    

The torus $Z_{\widetilde{\mathbf M}}$ splits over $E$, so all its cocharacters are defined over $E$.  The projection $Z_{\widetilde{\mathbf M}} \rightarrow Z_{\widetilde{\mathbf M}}/Z_{\widetilde{\mathbf G}}$ is also defined over $E$.  So we just need to check that if $\tau \in \Gal(\overline{F}/F)$, and $\tau|_{E} \neq 1_E$, then $\tau$ fixes the image of $\lambda(x)$ modulo $Z_{\widetilde{\mathbf G}}(\overline{F})$.  First, using the fact that $\tau( \sqrt{x}) = \pm \sqrt{x}$, that $(-1,-1,-1,-1) \in \mathbf K(\overline{F})$, and that $\tau$ acts on $\mathbf M^D(\overline{F})$ by $\tau.(x,y) = (\tau(y)^{-1}, \tau(x)^{-1})$ we get 

\begin{equation*}
    \begin{split}
         \tau.\lambda(x) & = \pm( \sqrt{x}, \frac{1}{\sqrt{x}}, \sqrt{x},  \frac{1}{\sqrt{x}}) \mathbf K(\overline{F}) \\ & = (\sqrt{x}, \frac{1}{\sqrt{x}}, \sqrt{x}, \frac{1}{\sqrt{x}}) \mathbf K(\overline{F})  
         \end{split}
         \end{equation*}
Next, $Z_{\widetilde{\mathbf G}}(\overline{F})$ embeds into $Z_{\widetilde{\mathbf M}}(\overline{F})$ as $(x,y) \mapsto (1,1,x,y)\mathbf K(\overline{F})$.  For $\lambda(x)$ modulo $Z_{\widetilde{\mathbf G}}(\overline{F})$ to be an $F$-rational point, it suffices to show that $\tau.\lambda(x)$ is congruent to $\lambda(x)$ modulo $Z_{\widetilde{\mathbf G}}(\overline{F})$.  And this is the case, using the element

\[ (1,1,\frac{1}{x},x) \mathbf K(\overline{F}) \]

and the fact that $\frac{\sqrt{x}}{x} = \frac{1}{\sqrt{x}}$ for any $x \in \overline{F}^{\ast}$ and any choice of square root of $x$.
\end{proof}

\subsection{The local coefficient as a partial Bessel integral}

We will need a nice collection of open compact subgroups $\overline{N}_{\kappa}, \kappa \in \mathbb Z$ of $\overline{\mathbf N}(F)$ for the proof of stability.  Note that $\overline{\mathbf N}(F)$, like $\mathbf N(F)$, identifies with the space of Hermitian matrices in $\operatorname{Mat}_n(E)$.

For $\pi$ an irreducible, admissible representation of $\mathbf M(F) = \GL_n(E)$ with central character $\omega_{\pi}: E^{\ast} = Z_{\mathbf M(F)} \rightarrow \mathbb{C}^{\ast}$, let $f$ be the conductor of $\omega_{\pi}|_{F^{\ast}}$.  Also, let $d$ be the conductor of the additive character $\psi$ of $F$, which was fixed once and for all.  We have a collection of open  compact neighborhoods of the identity in $\operatorname{Mat}_n(E)$ whose union is the entire space:

\[ X(\kappa) =  \begin{pmatrix} (\varpi_F)^{-\kappa} & (\varpi_F)^{-2\kappa} & (\varpi_F)^{-3\kappa} & \cdots \\ (\varpi_F)^{-2\kappa} &  (\varpi_F)^{-3\kappa} \\ (\varpi_F)^{-3\kappa} & & \ddots \\  \vdots\end{pmatrix} \]

where $\varpi_F$ is a uniformizer for $F$, and $(\varpi_F) = \varpi_F \mathcal O_E$.  Of course $(\varpi_F) = \varpi_E \mathcal O_E$ if $E/F$ is not ramified, and $(\varpi_F) = \varpi_E^2 \mathcal O_E$ if $E/F$ is ramified.  Equivalently,

\[ X(\kappa) = \{ x \in \operatorname{Mat}_n(E) : x_{ij} \in (\varpi_F)^{-\kappa( i + j -1)} \} \]

We let 

\[ \overline{N}_{ \kappa} = \{ X \in \overline{\mathbf N}(F) : \varpi_F^{-d-f}X \in X(\kappa) \} \]

Although the indexing of our open compact subgroups depends on $\pi$, the total sequence does not.  We will let $\varphi_{\kappa}$ be the characteristic function of $X(\kappa)$.

\newtheorem{G1}{Lemma}[subsection]

\begin{G1}For $t \in F^{\ast}$, $\alpha^{\vee}(t) \overline{N}_{\kappa} \alpha^{\vee}(t)^{-1}$ only depends on $|t|_F$.

\end{G1}

\begin{proof} Let $t \in F^{\ast}$.  Then $\alpha^{\vee}(t)$ is an element of $Z_{\widetilde{\mathbf M}(F)}$ which is only well defined up modulo $Z_{\widetilde{\mathbf G}(F)}$.  However, conjugation by $\alpha^{\vee}(t)$ is well defined, and coincides with conjugation by the $E$-rational point
\[ \begin{pmatrix} tI_n \\ & I_n \end{pmatrix} \]
so we see that conjugation by $\alpha^{\vee}(t)$ of an element $X$ of $\overline{\mathbf N}(F)$, identified with a Hermitian matrix, produces the Hermitian matrix $t^{-1}X$.  
\end{proof}

Recall that in Section 3.5, we choose once and for all an extension of each character $\omega$ of $Z_{\mathbf M(F)}$ to a character $\tilde{\omega}$ of $Z_{\widetilde{\mathbf M}(F)}$.  Let $\pi$ be an supercuspidal representation of $\mathbf M(F) = \GL_n(E)$.  Let $\omega$ be the central character of $\pi$.  By the results of Section 3.5, there exists a generic representation $\widetilde{\pi}$ of $\widetilde{\mathbf M}(F)$, having central character $\tilde{\omega}$, such that $\pi$ is isomorphic to a subrepresentation of $\widetilde{\pi}|_{\mathbf M(F)}$, and the local coefficient $C_{\chi}(s, \widetilde{\pi})$ (relative to $\widetilde{\mathbf M}(F)$ inside $\widetilde{\mathbf G}(F)$) is equal to $C_{\chi}(s,\pi)$.  

The central character $\tilde{\omega}_{\widetilde{\pi}_s}$ of $\widetilde{\pi}_s = \widetilde{\pi}  q^{\langle s \tilde{\alpha}, H_{\widetilde{\mathbf M}}(-) \rangle}$ is equal to $\tilde{\omega}  q^{\langle s \tilde{\alpha}, H_{\mathbf M}(-) \rangle}$.  Let us first compute the character $\tilde{\omega}(\dot w_0 \tilde{\omega}^{-1})$ (Section 6 of [Sh1]) of $F^{\ast}$ which is defined by
\[ \tilde{\omega}(\dot w_0 \tilde{\omega}^{-1})(t) = \tilde{\omega}(\alpha^{\vee}(t)\dot w_0^{-1}\alpha^{\vee}(t)\dot w_0)\]
This is well defined as a character of $F^{\ast}$, even though $\alpha^{\vee}(t) \in Z_{\widetilde{\mathbf M}}(F)$ is only well defined modulo $Z_{\widetilde{\mathbf G}}(F)$.

\newtheorem{G2}[G1]{Lemma}

\begin{G2} Let $t \in F^{\ast}$.  Identifying $\mathbf M(F) = \GL_n(E)$, we have

\[ \tilde{\omega}(\dot w_0 \tilde{\omega}^{-1})(t) = \omega(tI_n) \]

\end{G2}

In particular, $\tilde{\omega}(\dot w_0 \tilde{\omega}^{-1})$ does not depend on the choice of character $\tilde{\omega}$ extending $\omega$, and $\tilde{\omega}(\dot w_0 \tilde{\omega}^{-1})$ is ramified if and only if $\omega: E^{\ast} = Z_{\mathbf M(F)} \rightarrow \mathbb{C}^{\ast}$ is ramified.

\begin{proof} Choose any square root $\sqrt{t} \in \overline{F}^{\ast}$ of $t$, and define

\[ z = [\begin{pmatrix} \frac{1}{\sqrt{t}} I_n \\ & \sqrt{t}I_n \end{pmatrix}, (\sqrt{t}, \frac{1}{\sqrt{t}})] \mathbf K(\overline{F}) \in Z_{\widetilde{\mathbf M}}(\overline{F}) \]

  Let $z_0 \in Z_{\widetilde{\mathbf M}(F)}$ be a representative modulo $Z_{\widetilde{\mathbf G}(F)}$ of $\alpha^{\vee}(t)$.  The definition of $\tilde{\omega}(\dot w_0 \tilde{\omega}^{-1})(t)$ is

\[ \tilde{\omega}(z_0\dot w_0^{-1}z_0 \dot w_0) \]

By the definition of $\alpha^{\vee}(t)$, there exists a $g \in Z_{\mathbf G}(\overline{F})$ such that $z = z_0g$.  Then $z_0\dot w_0^{-1}z_0 \dot w_0 = z \dot w_0^{-1}z\dot w_0$, with 

\[ z \dot w_0^{-1}z \dot w_0 = [\begin{pmatrix} tI_n \\ & t^{-1}I_n \end{pmatrix}, (1,1)]\mathbf K(\overline{F}) \]

which lies in $Z_{\mathbf M(F)}$ and identifies with the matrix $tI_n$ in the center of $\GL_n(E)$.\end{proof}

Let $\mathbf Z^0$ be the isomorphic image of $F^{\ast}$ under the homomorphism $\alpha^{\vee}$, and let $z \in \mathbf Z^0$.  Let $n$ be an element of the open dense subset $W$ of $\mathbf N(F)$ defined in Proposition 3.4.1, so that the stabilizer $\mathbf U_{\mathbf M,n}(F)$ of $n$ under conjugation by $\mathbf U_{\mathbf M}(F)$ is trivial.  Write $\dot w_0^{-1}n = mn'\bar{n}$ as in Lemma 3.3.4.  We remark that we have $\mathbf U_{\mathbf M,n}(F) = \mathbf U_{\mathbf M,m}(F)' = 1$, where $\mathbf U_{\mathbf M,m}'$ is as in Section 3 of [Sh1].  Hence Assumption 4.1 of [Sh1] is satisfied.  

Let $\tilde{W}$ be an element of the Whittaker model of $\widetilde{\pi}_s$ with $\tilde{W}(e) = 1$.  We define the partial Bessel integral
\[ j_{\overline{N}_{\kappa}, \tilde{W}}(m,z) = \int\limits_{\mathbf U_{\mathbf M}(F)} \tilde{W}(mu) \phi_{\kappa}(u^{-1}z\bar{n}z^{-1}u)\overline{\chi(u)}\, du \]
where $\phi_{\kappa}$ is the characteristic function of $\overline{N}_{\kappa}$.  We can write $\dot{w_0}^{-1}\bar{n}\dot w_0 = n_1$ for $n_1 \in \mathbf N(F)$, so that
\[ n_1 = \prod\limits_{\beta \in \Delta} \mathbf x_{\beta}(x_{\beta}) \, n''\]
for $x_{\beta} \in \overline{F}$ and $n''$ in the derived group of $\mathbf N(F)$.  The element $x_{\alpha} := x_{e_n - e_{2n}}$ lies in $F$, because the character $e_n - e_{2n}$ of $\mathbf T$ is defined over $F$.

Let us compute the matrices $m, n', \bar n$ and the element $x_{\alpha}$ for special $n \in \mathbf N(F)$.  Recall that both $\mathbf N(F)$ and $\overline{\mathbf N}(F)$ identify naturally with the space of Hermitian matrices with entries in $E$.  And by Lemma 3.3.4, $\dot w_0^{-1}n \in \mathbf P(F) \overline{\mathbf N}(F)$ if and only if the Hermitian matrix corresponding to $n$ is invertible.

\newtheorem{G3}[G1]{Lemma}

\begin{G3} Let $r' = \textrm{diag}(1, r_2', ... , r_n')$ be a diagonal matrix with entries in $F^{\ast}$.  Let 

\[ n = \begin{pmatrix} I_n & r' \\ & I_n \end{pmatrix} \in \mathbf N(F) \]

Then $\dot w_0^{-1}n = mn'\bar n$ with $m \in \mathbf M(F), n' \in \mathbf N(F), \bar n \in \overline{\mathbf N}(F)$, where

(i): If we identify $m$ with a matrix in $\GL_n(E)$, then $m = (-1)^{n-1}Jr'^{-1}$.

(ii): If we identify $\bar n$ with a Hermitian matrix, then $\bar n = r'^{-1}$.

(iii): The element $x_{\alpha} \in F^{\ast}$ corresponding to $n$ above is $-1$.
\end{G3}

Here $J$ is as in Lemma 3.3.2.  
\begin{proof} (i) and (ii) are immediate from the Lemma 3.3.4.  For (iii), we first need to compute $n_1 = \dot w_0^{-1} \bar n \dot w_0$.   We have 

\[ n_1 = \dot w_0^{-1} \bar n \dot w_0= (-1)^n \begin{pmatrix} & J \\ -J \end{pmatrix} \begin{pmatrix} I_n \\ r'^{-1} & I_n \end{pmatrix} \begin{pmatrix} & J \\ -J \end{pmatrix} = \begin{pmatrix} I_n & (-1)^n Jr'^{-1}J \\ & I_n \end{pmatrix}\]

where the lower right entry $x_{\alpha}$ of $(-1)^nJa'^{-1}J$ is easily seen to be $-1$.

\end{proof}

Now let $f$ be a matrix coefficient of $\pi$, and let 

\[ W^f(m) = \int\limits_{\mathbf U_{\mathbf M}(F)} f(um) \overline{\chi(u)} \, du = \int\limits_U f(xm) \overline{\chi(x)} \, dx \]

where $U$ is the group of upper triangular unipotent matrices in $\GL_n(E)$.  Then $W^f$ lies in the Whittaker model of $\pi$.  We may choose $f$ so that $W^f(e) = 1$.  By Lemma 3.5.2, $W^f$ extends to a function $\widetilde{\mathbf M}(F) \rightarrow \mathbb{C}$ in the Whittaker model of $\widetilde{\pi}$.  Also call this extension $W^f$.  Then

\[ W(\tilde m) := q^{\langle s \alpha, H_{\widetilde{\mathbf M}}(\tilde m)\rangle } W^f(\tilde m) \]

lies in the Whittaker model of $\widetilde{\pi}_s = \widetilde{\pi} q^{\langle s \tilde{\alpha}, H_{\widetilde{\mathbf M}}(-) \rangle}$.  Now that we have defined our lengthy notation, we can state Shahidi's local coefficient formula for $C_{\chi}(s,\tilde{\pi})$.

\newtheorem{G4}[G1]{Theorem}

\begin{G4} Assume that the central character of $\pi$ is ramified (so that the central character of $\widetilde{\pi}$ is also ramified).  Then the local coefficient $C_{\chi}(s,\pi) = C_{\chi}(s,\tilde{\pi})$ is equal to
\[\gamma(s) \int\limits_{\mathbf Z^0 \mathbf U_{\mathbf M}(F) \backslash \mathbf N(F)} j_{\overline{N}_{\kappa},W}(m) \widetilde{\omega}_{\widetilde{\pi}_s}(\dot w_0\widetilde{\omega}_{\widetilde{\pi}_s})(x_{\mathbf a}) q^{\langle s \tilde{\alpha}  + \rho,H_{\widetilde{\mathbf M}}(m) \rangle} d \dot n \]
for all $\kappa \in \mathbb Z$, where $\gamma(s)$ is a function depending only on the central character of $\pi$.  Here
\[j_{\overline{N}_{\kappa},W}(m) = j_{\overline{N}_{\kappa},W}(m, \alpha^{\vee}(\varpi^{d+f} x_{\alpha}))\]
where $d$ and $f$ are the conductors of $\psi$ and $\omega|_{F^{\ast}}$.
\end{G4}

We have already identified the quotient space $\mathbf Z^0 \mathbf U_{\mathbf M}(F) \backslash \mathbf N(F)$ in Section 3.4 with the torus $R'$.  Conjugation by $\mathbf Z^0$ coincides with the action of $F^{\ast}$ given there, and the measure $\dot n$ is the measure $dr'$.

\begin{proof} This is Theorem 6.2 of [Sh1].  The only new claim we make is that this is valid for all $\kappa \in \mathbb Z$, rather than for $\kappa$ sufficiently large.  In the generality in which Shahidi proved his formula, he considered the partial Bessel function $j_{\overline{N}_0,W}(m)$ for compact open subgroups $\overline{N}_0$ of $\overline{\mathbf N}(F)$ for which $\alpha^{\vee}(t)\overline{N}_0\alpha^{\vee}(t)^{-1}$ only depends on $|t|_F$.  In general, $\overline{\mathbf N}(F)$ has arbitrary such large subgroups, but they need not be arbitrarily small.  However, our subgroups $\overline{N}_{\kappa}$ can be chosen arbitrarily small, allowing us to modify the proof of Shahidi's theorem to hold for arbitrary $\overline{N}_{\kappa}$. 

Our element $W$ in the Whittaker model of $\tilde{\pi}_s$ is a map $g \mapsto \lambda(\tilde{\pi}v)$ for some vector $v$ in the space $V$ of $\tilde{\pi}$ and a Whittaker functional $\lambda$, satisfying $\lambda(v) = 1$.  Take $\kappa_0$ to be an integer sufficiently small so that $\overline{N}_{\kappa_0}$ is contained in the kernel of the character $\chi'(n^-) =  \chi(\dot{w}_0^{-1} n^- \dot w_0)$ of $\overline{\mathbf N}(F)$.  Let $f: \overline{\mathbf N}(F) \rightarrow V$ be $v$ times the characteristic function of $\overline{N}_{\kappa_0}$ divided by the measure of $\overline{N}_0$.  Now $f$ extends to an element of $I(s, \tilde{\pi})$, vanishing off of $\mathbf P(F) \overline{N}_{\kappa_0}$ and we have

\[ v = \int\limits_{\overline{N}_{\kappa}} f(n^-) \overline{\chi'(n^-)} \, dn^-\]

for all $\kappa \geq \kappa_0$.  Now for $h$ in the induced space $I(s,\pi)$ or $I(-s,w_0(\pi))$, the Whittaker functional

\[ \int\limits_{\mathbf N(F)} \langle R_{\dot w_0^{-1}}(h)(\dot w_0n), \lambda \rangle \overline{\chi(n)} \, dn = \int\limits_{\overline{\mathbf N}(F)} \langle f(\bar n_1), \lambda \rangle \overline{\chi'(\bar n_1)} \, d \bar n_1  \]

is really defined to be an integral over a suitable open compact subgroup $N_0$ of $\mathbf N(F)$ (or an open compact subgroup $\overline{N}_0 = w_0 N_0 w_0^{-1}$ or $\overline{\mathbf N}(F)$).  The subgroup $\overline{N}_0$ will depend on $h$ but not on $s$, and it will have the property that the integral remains the same if $\overline{N}_0$ is replaced by any open compact subgroup of $\mathbf N(F)$ containing $\overline{N}_0$ (Theorem 3.4.7 of [Sh2]).  Let $\overline{N}_0$ be an open compact subgroup which works for $R_{\dot w_0^{-1}}(f)$ and for $A(s,\pi)(R_{\dot w_0^{-1}}(f))$.  Let $\overline{N}_1$ be an open compact subgroup of $\mathbf N(F)$ containing $\overline{N}_0$ and $\overline{N}_{\kappa}$, for some fixed $\kappa \geq \kappa_0$.  Then by comparing integration over $\overline{N}_0, \overline{N}_1$, and $\overline{N}_{\kappa}$, we see from our choice of $f$  that our Whittaker functionals evaluated at $R_{\dot w_0^{-1}}(f)$ and $A(s,\pi)(R_{\dot w_0^{-1}}(f))$ can in fact be obtained by integration over $\overline{N}_{\kappa}$.  

The rest of the proof follows exactly as it was written in [Sh1].  Since $\kappa_0$ could be chosen arbitrarily small, we see that Shahidi's local coefficient formula will now hold for all $\kappa \in \mathbb Z$.
\end{proof}

\subsection{The partial Bessel integral}

Let $G = \GL_n(E) = \mathbf M(F)$, $B$ and $A$ the usual Borel subgroup and maximal torus of $G$, and $U$ the unipotent radical of $B$.  Let $A' = \{ (1,a_2, ... , a_n) \in A\}$, so that $A$ is the direct product of $A'$ and the center $Z$ of $G$.  If $a \in A$, let $a'$ be the element of $A'$ obtained by ``stripping off the center'' of $a$, so that $a = a'z$ for $z \in Z$.

Let $W(G)$ be the Weyl group of $G$.  For $w \in W(G)$, we keep our Weyl group representatives $\dot w$ from Section 3.3.  For $g \in G$, there is a unique $w \in W(G)$ such that $g$ lies in the Bruhat cell $C(w) = BwB$.  If $U_w^-$ is the subgroup of $U$ directly spanned by the root subgroups of those roots which are made negative by $w$, then we can write $g$ uniquely as $u_1\dot wau_2$ for $u_1 \in U, a \in A,$ and $u_2 \in U_w^-$.  For a subset $S$ of $G$ containing $Z$, define $\mathscr C_c^{\infty}(S;\omega)$ to be the space of locally constant functions $f: S \rightarrow \mathbb C$ which are compactly supported modulo $Z$ and which satisfy $f(zg) = \omega(z)f(g)$ for $z \in Z$ and $g \in G$.  

Let $f \in \mathscr C_c^{\infty}(G;\omega)$.  For example, $f$ could be a matrix coefficient of $\pi$, because $\pi$ is supercuspidal.  Define a map $W^f: G \rightarrow \mathbb C$ by

\[ W^f(g) = \int\limits_U f(xg) \overline{\chi(x)} \, dx\]

where $\chi$ is the restriction to $U = \mathbf U_{\mathbf M}(F)$ of our generic character of $\mathbf U(F)$.   This integral converges absolutely.  Now $U$ acts on $G$ on the right by $g.u = \dot w_G \, ^t \overline{u} \dot w_G^{-1}gu$, where $w_G = w_{\ell}^{\theta}$ is the long element of $G$.  Let $U_g$ be the stabilizer of a given $g \in G$ under this action, and let $\varphi$ be the characteristic function of an open compact subset of $\operatorname{Mat}_n(E)$.  We define the partial Bessel integral $B_{\varphi}^G(g,f)$ by

\[ B_{\varphi}^G(g,f) = \int\limits_{U_g \backslash U} W^f(ug) \varphi(^t \overline{u} \dot w_G^{-1}g'u) \overline{\chi(u)} \, du \]
The integral converges absolutely, on account of the fact that $f$ is compactly supported modulo $Z$, and that for a $p$-adic field $k$, the $k$-points of orbits of unipotent groups over $k$ acting on affine $k$-varieties are closed. \\

We shall now rewrite the formula in Theorem 3.7.4.  By the results of Section 3.4, may identify $\mathbf Z^0 \mathbf U_{\mathbf M}(F) \backslash \mathbf N(F)$ with the space $R'$ of matrices of the form $\textrm{diag}(1,r_2', ... , r_n')$ with $r_i' \in F^{\ast}$.  The measure $d \dot n = dr'$ is then the measure

\[ dr' = \prod\limits_{i=2}^n |r_i'|_F^{2i-1} d^{\ast} r_i' \]

where $d^{\ast} r_i'$ is the usual Haar measure $\frac{dr_i}{|r_i|_F}$ on $R' = (F^{\ast})^{n-1}$.  If $n \in \mathbf N(F)$ corresponds to $r'$, i.e.

\[ n = \begin{pmatrix} I_n & r' \\ & I_n \end{pmatrix} \]

then writing $\dot w_0^{-1}n  = mn' \bar n$, we have $m = (-1)^{n-1}Ja'^{-1} = (-1)^{n-1} \dot w_G a'^{-1}$, $\bar n = a'^{-1}$, and $x_{\alpha} = -1$ (Lemma 3.7.3).  Recall that the matrix $J$ of Section 3.3 is equal to $\dot w_G$,  the representative of the long Weyl group element in $\GL_n(E)$.

\newtheorem{S1}{Lemma}[subsection]

\begin{S1} With $n, m, \bar n$ as above, we have 

(i): $\widetilde{\omega}_{\widetilde{\pi}_s}(\dot w_0\widetilde{\omega}_{\widetilde{\pi}_s})(x_{\alpha}) = \pm 1$

(ii): $q^{\langle s \tilde{\alpha} + \rho, H_{\widetilde{\mathbf M}}(m) \rangle } =  \prod\limits_{i=2}^n |r_i'|^{-(s+n)}$

\end{S1}

\begin{proof} (i) is on account of the fact that $x_{\alpha} = -1$.  (ii) follows from Lemma 3.3.1, and the fact that the restriction of  $H_{\widetilde{\mathbf M}}$ to $\mathbf M(F)$ is $H_{\mathbf M(F)}$.  

\end{proof}

Now for $j_{\overline{N}_{\kappa},W}(m) = j_{\overline{N}_{\kappa},W}(m, \alpha^{\vee}(\varpi^{d+f} x_{\alpha})$, where $d$ and $f$ are the conductors of $\psi$ and $\omega|_{F^{\ast}}$.  Consider first $\alpha^{\vee}(\varpi^{d+f} x_{\alpha}) = \alpha^{\vee}(-\varpi^{d+f})$.  Since $\alpha^{\vee}(t) \overline{N}_{\kappa} \alpha^{\vee}(t)^{-1}$ only depends on $|t|_F$, we can ignore the sign and just work with $z = \alpha^{\vee}(\varpi^{d+f})$. We have 

\[ j_{\overline{N}_{\kappa},W}(m) = \int\limits_{\mathbf U_{\mathbf M}(F)} W(mu) \phi_{\kappa}(u^{-1}z\bar n z^{-1}u) \overline{\chi(u)} \, du \]

If we identify $\overline{\mathbf N}(F)$ with the Hermitian matrices in $\operatorname{Mat}_n(E)$, and $\mathbf U_{\mathbf M}(F)$ with the group $U$ of upper triangular unipotent matrices with entries in $E$, then $u^{-1}z \bar n z^{-1} u$ is simply

\[ \varpi^{-d-f}{^t}\overline{u} r'^{-1}u = \omega^{-d-f} {^t}\overline{u} \dot w_G^{-1} \dot w_G r'^{-1}u \]

Now

\begin{equation*}
    \begin{split}
        W(mu) = W((-1)^{n-1}\dot w_G r'^{-1}u)&  = \pm q^{\langle s \tilde{\alpha}, H_{\mathbf M}(\dot w_Gr'^{-1}u) \rangle} W^f( \dot w_G r'^{-1}u) \\
        & = \pm \prod\limits_{i=2}^n |r_i'|_F^{-s} W^f(\dot w_Gr'^{-1}u)
    \end{split}
\end{equation*}

Observing that $\varpi^{-d-f} {^t}\overline{u}\dot w_G^{-1} \dot w_G r'^{-1}u$ lies in $\overline{N}_{\kappa}$ if and only if \[^t\overline{u} r'^{-1}u = {^t}\overline{u} \dot w_G^{-1} \dot w_G r'^{-1}u\]lies in $X(\kappa)$, we then have

\begin{equation*}
    \begin{split}
        j_{\overline{N}_{\kappa},W}(m) & = \pm \prod\limits_{i=2}^n |r_i'|^{-s} \int\limits_U W^f(\dot w_G r'^{-1}u) \varphi_{\kappa}(^t\overline{u} \dot w_G^{-1} \dot w_G r'^{-1}u) \overline{\chi(u)} \, du \\
        & = \pm \prod\limits_{i=2}^n |r_i'|^{-s} B_{\varphi_{\kappa}}^G(\dot w_Gr'^{-1},f)
    \end{split}
\end{equation*}

where $\varphi_{\kappa}$ is the characteristic function of $X(\kappa)$ (not $\overline{\mathbf N}_{\kappa}$).  Absorbing the $\pm 1$ into $\gamma$, and combining everything together, we get

\[ C_{\chi}(s,\pi)^{-1} = \gamma(s) \int\limits_{R'} B_{\varphi_{\kappa}}^G(\dot w_Gr'^{-1},f) \prod\limits_{i=2}^n |r_i'|^{-2s+n+2i-1} d^{\ast}r_i' \]

Finally making the change of variables $r' \mapsto r'^{-1}$, we arrive at the following

\newtheorem{S2}[S1]{Proposition}

\begin{S2} Let $\pi$ be an irreducible, supercuspidal representation of $\mathbf M(F) = \GL_n(E)$ with central character $\omega$.  Let $f$ be a matrix coefficient of $\pi$.  If $\omega$ is ramified, then there exists a function $\gamma = \gamma_{\omega}$ depending only on the central character $\omega$ and not on $\pi$, such that 

\[ C_{\chi}(s,\pi)^{-1} = \gamma(s) \int\limits_{R'} B_{\varphi_{\kappa}}^G(\dot w_Gr',f) \prod\limits_{i=2}^n |r_i'|^{2s-n-2i+1} d^{\ast}r_i'  \]

for all $\kappa \in \mathbb Z$. \end{S2}

Assume that $\pi$ is an irreducible supercuspidal representation of $\mathbf M(F) = \GL_n(E)$, whose central character $\omega$ is not necessarily ramified.  Assume that $\eta$ is a characer of $E^{\ast}$. Let $f$ be a matrix coefficient of $\pi$ such that $W^f(e) = 1$.  Then $f(g)\eta(g)$ is a matrix coefficient of $\pi  \eta$ with $W^{f\eta}(e) = 1$.  Clearly $B_{\varphi_{\kappa}}^G(\dot w_G r',f\eta) = \eta(r') B_{\varphi_{\kappa}}^G(\dot w_Gr',f)$, so applying Proposition 3.8.2 to $\pi \eta$, we get:

\newtheorem{S3}[S1]{Proposition}

\begin{S3} Let $\pi$ be an irreducible, supercuspidal representation of $\mathbf M(F) = \GL_n(E)$ with central character $\omega$.  Let $f$ be a matrix coefficient of $\pi$.  For all characters $\eta$ of $E^{\ast}$ such that $\omega \eta^n$ is ramified, there exists a function $\gamma = \gamma_{\eta}$ depending only on the character $\omega \eta^n$ and not on $\pi$, such that

\[ C_{\chi}(s,\pi \eta)^{-1} = \gamma_{\eta}(s) \int\limits_{R'} \eta(r') B_{\varphi_{\kappa}}^G(\dot w_Gr',f) \prod\limits_{i=2}^n |r_i'|^{2s-n-2i+1} d^{\ast}r_i'  \]
for all $\kappa \in \mathbb Z$.
\end{S3}

\subsection{Bessel function asympotics}

What happens next is a detailed study of the asymptotics of the partial Bessel integrals $B_{\varphi}^G$.  This was done in a slightly different setting by Cogdell, Shahidi, and Tsai in Section 5 of [CoShTs].  Their argument encompasses forty pages of hard analysis.  Fortunately, our calculations turn out almost entirely identical to theirs.  Their field is $F$, our field is $E$.  Their use of the transpose $^tg$ must be replaced by the conjugate transpose $^t\overline{g}$, where $\overline{g}$ is the application of the nontrivial element of $\Gal(E/F)$.  We cite their main result, Proposition 3.9.1 (Proposition 5.7 of their paper), referring to their paper for an almost word for word identical proof.

Let $B(G)$ be the set of $w \in W(G)$ of the form $w_G w_M^{-1}$, where $M$ is a standard Levi subgroup of $G$, and $w_G$ and $w_M$ are the long elements of $G$ and $M$.  If $w, w' \in B(G)$, we have a notion $d_B(w,w')$ of Bessel distance ([ShCoTs], 5.1.4).  For $w' \in B(G)$, we set
\[ \Omega_{w'} = \bigcup\limits_{w \leq w'} BwB \]
where $\leq$ is the Bruhat order.  This union is open in $G$.  

For $w \leq w'$ in $B(G)$, with corresponding Levi subgroups $M \subseteq M'$, let $A_w$ and $A_{w'}$ be the centers of $M$ and $M'$.  Define the transverse torus $A_w^{w'} = A_w \cap M_{\textrm{der}}'$.  The group $A_w^{w'}A_{w'}$ is open and finite index in $A_w$, with $A_w^{w'} \cap A_{w'}$ finite.

We fix once and for all an auxiliary function $f_0 \in \mathscr C_c^{\infty}(G;\omega_{\pi})$ with $W^{f_0}(e) =1$.

\newtheorem{K1}{Proposition}[subsection]

\begin{K1} (Proposition 5.7 of [CoShTs]) Let $f$ be a matrix coefficient of $\pi$ with $W^f(e) = 1$.  Then there exists a function $f_{1,e} \in \mathscr C_c^{\infty}(G;\omega_{\pi})$, and for each $w' \in B(G)$ with $1 \leq d_B(w',e)$ a function $f_{1,w'} \in \mathscr C_c^{\infty}(\Omega_{w'};\omega_{\pi})$ such that:

(i): For all sufficiently large $\varphi = \varphi_{\kappa}$ (which is to say, for sufficiently large $\kappa$), we have

\[ B_{\varphi}^G(\dot w_Ga,f) = B_{\varphi}^G(\dot w_Ga,f_{1,e}) + \sum\limits_{1 \leq d_B(w',e)} B_{\varphi}^G(\dot w_Ga,f_{1,w'}) \]

for all $a \in A$.

(ii): $B_{\varphi}^G(\dot w_Ga,f_{1,e})$, as a function of $a$, depends only on the auxiliary function $f_0$ and on the central character $\omega_{\pi}$.

(iii): For sufficiently large $\varphi$, and for each $w' \in B(G)$ with $1 \leq d_B(w',e)$, we have that $B_{\varphi}^G(\dot w_G a, f_{1,w'}) = \omega_{\pi}(z) B_{\varphi}^G(\dot w_G bc', f_{1,w'})$ vanishes for $a$ outside $A_{w_G}^{w'}A_{w'}$ is uniformly smooth as a function of $c' \in A_{w'}'$.  
\end{K1}

Part (iii) of Proposition 3.9.1 requires some explanation.  The product $A_{w_G}^{w'}A_{w'}$ is open and of finite index in $A = A_{w_G}$, and one considers those $a \in A$ which decompose (in finitely many ways) as a product $zbc'$, for $b \in A_e^{w'}$, $z \in Z$, and $c' \in A_{w'}'$, where $A_{w'}' = A_{w'} \cap A'$.  A function $h$ of $c' \in A_{w'}'$ is called uniformly smooth if there exists an an open compact subgroup $H$ of $A_{w'}'$ such that $h(c_0c') = h(c')$ for all $c_0 \in H$ and $c' \in A_{w'}'$.  

\bigskip 

Now we complete the proof of Proposition 3.3.3, and hence the proof of Proposition 2.5.8.  We are given two irreducible supercuspidal representations $\pi_1$ and $\pi_2$ of $\mathbf M(F) = G$ with the same central character $\omega$.  Let $f_1, f_2 \in \mathscr C_c^{\infty}(G;\omega)$ be matrix coefficients of $\pi_1$ and $\pi_2$ with $W^{f_1}(e) = W^{f_2}(e) = 1$.  Fix an auxiliary function $f_0 \in \mathscr C_c^{\infty}(G;\omega)$ with $W^{f_0}(e) = 1$.  By Proposition 3.9.1, there exist functions $f_{1,e}, f_{2,e} \in \mathscr C_c^{\infty}(G;\omega)$ and for each $w' \in B(G)$ with $1 \leq d_B(w',e)$, functions $f_{1,w'}, f_{2,w'} \in \mathscr C_c^{\infty}(\Omega_{w'};\omega)$ such that the following hold:

\begin{itemize}
    \item For sufficiently large $\varphi = \varphi_{\kappa}$ (that is, for sufficiently large $\kappa$),
    
    \[ B_{\varphi}^G(\dot w_Ga',f_i) = B_{\varphi}^G(\dot w_Ga',f_{i,e}) + \sum\limits_{1 \leq d_B(w',e)} B_{\varphi}^G(\dot w_Ga', f_{i,w'}) \]
    
    for $i = 1, 2$ and all $a \in  A$.
    
    \item Each $B_{\varphi}^G(\dot w_Ga', f_{i,e})$ as a function of $a'$, only depends on the auxiliary function $f_0$ and on $\omega$.
    
    \item For sufficiently large $\varphi$, and each $w' \in B(G)$, with $1 \leq d_B(w',e)$, $B_{\varphi}^G(\dot w_Gbc', f_{i,w'})$ vanishes for $a$ outside $A_{w_G}^{w'}A_{w'}$ and is uniformly smooth as a function of $c' \in A_{w'}'$.
\end{itemize}

For the groups $A,A', A_w^{w'}$ etc. defined above, we define the subgroups $R, R', R_{w}^{w'}$ etc. as the corresponding ones with entries in $F$.  Note that $R'$ coincides with its earlier definition in Section 3.4.  We have for each character $\eta$ of $E^{\ast}$ such that $\omega \eta^n$ is ramified, a function $\gamma_{\eta}(s)$ depending only on the character $\omega \eta^n$ such that 

\[ C_{\chi}(s,\pi_i  \eta)^{-1} = \gamma_{\eta}(s) \int\limits_{R'}\eta(r') B_{\varphi_{\kappa}}^G(\dot w_Gr',f_i) \prod\limits_{i=2}^n |r_i'|^{2s-n-2i+1} d^{\ast}r_i'  \]

We will drop $\kappa$ from the notation and write $\varphi$ instead of $\varphi_{\kappa}$.  The difference $C_{\chi}(s,\pi_1  \eta)^{-1} - C_{\chi}(s,\pi_2 \eta)^{-1}$ is equal to 

\[ \gamma_{\eta}(s) \int\limits_{R'} ( B_{\varphi}^G(\dot w_Gr',f_1) - B_{\varphi}^G(\dot w_Gr',f_2)) \eta(r')\prod\limits_{i=2}^n |r_i'|^{2s-n-2i+1} d^{\ast}r_i' \]

For sufficiently large $\varphi$, we will have 

\begin{equation*}
    \begin{split}
        B_{\varphi}^G(\dot w_Gr',f_1) - B_{\varphi}^G(\dot w_Gr',f_2) & = B_{\varphi}^G(\dot w_Gr',f_{1,e}) - B_{\varphi}^G(\dot w_Gr',f_{2,e}) \\
        & + \sum\limits_{1 \leq d_B(w',e)} B_{\varphi}^G(\dot w_Gr', f_{1,w'}) - B_{\varphi}^G(\dot w_Gr', f_{2,w'})
    \end{split}
\end{equation*} 

Since $B_{\varphi}^G(\dot w_Gr',f_{1,e})$ and $B_{\varphi}^G(\dot w_Gr',f_{2,e})$ each only depend on the central character $\omega$ and on the auxiliary function $f_0$, these will cancel.  To complete the proof of the theorem, it suffices to show that for each $w' \in B(G)$ with $1 \leq d_B(w',e)$, we have

\[ \int\limits_{R'} (B_{\varphi}^G(\dot w_Gr', f_{1,w'}) - B_{\varphi}^G(\dot w_Gr', f_{2,w'})) \eta(r') \prod\limits_{i=2}^n |r_i'|^{2s-n-2i+1} d^{\ast}r' = 0 \]

Now each $B_{\varphi}^G(\dot w_Ga', f_{i,w'})$ vanishes for $a' \not\in A_{w_G}^{w'} A_{w'}$.  Since $A_{w_G}^{w'} A_{w'}$ is open in $A$, $A_{w_G}^{w'} A_{w'} \cap R'$ is open in $R'$, so the given integral is equal to integration over $A_{w_G}^{w'} A_{w'} \cap R'$.  We write

\[ \int\limits_{A_{w_G}^{w'} A_{w'} \cap R'} d^{\ast}r_i' = \int\limits_{A_{w_G}^{w'} A_{w'} \cap R'/R_{w'}'} \int\limits_{R_{w'}'} \, d^{\ast}c_i' d \bar x \]

For each $\bar x \in A_{w_G}^{w'} A_{w'} \cap R'/R_{w'}'$, we choose a representative $bd$ with $b \in A_{w_G}^{w'}$ and $d \in A_{w'}$.  Write $d= zd_1'$ for $z \in Z_M$ and $d_1' \in A_{w'}'$.  The integral which we are to show is zero is equal to 

\begin{equation*}
    \begin{split}
        \int\limits_{A_{w_G}^{w'} A_{w'} \cap R'/R_{w'}'} \prod\limits_{i=2}^n |b_id_i|^{2s-n-2i+1} \eta(bd) \omega(z) \\
    \int\limits_{R_{w'}'} (B_{\varphi}^G(\dot w_Gbd'c', f_{1,w'}) - B_{\varphi}^G(\dot w_Gbd'c', f_{2,w'})) \eta(c') \prod\limits_{i=2}^n |c_i'|^{2s-n-2i+1} \, d^{\ast} c_i' \, d \bar x
    \end{split}
\end{equation*}

By uniform smoothness, there exists an open compact subgroup $H$ of $A_{w'}'$, depending on $b$, such that for all $c'' \in A_{w'}'$, and all $c''' \in H$, 

\[B_{\varphi}^G(\dot w_Gbc''c''', f_{i,w'}) = B_{\varphi}^G(\dot w_Gbc''c''', f_{i,w'}) \]

It follows that, for fixed $b \in A_{w_G}^{w'},$ and $d' \in A_{w'}'$, $B_{\varphi}^G(\dot w_Gbd'c', f_{i,w'})$ is uniformly smooth with respect to $c' \in R_{w'}'$ (taking $H \cap R_{w'}'$ as the open compact subgroup).  This is to say,

\[B_{\varphi}^G(\dot w_Gbd'c'c'', f_{i,w'}) = B_{\varphi}^G(\dot w_Gbd'c', f_{i,w'})\]

for all $c' \in R_{w'}'$, and $c'' \in R_{w'} \cap H$.  Taking $\eta$ to be sufficiently highly ramified so that it is nontrivial on some $c'' \in R_{w'} \cap H$, we see by changing $c' \mapsto c''^{-1}c'$ that
\[\int\limits_{R_{w'}'} (B_{\varphi}^G(\dot w_Gbd'c', f_{1,w'}) - B_{\varphi}^G(\dot w_Gbd'c', f_{2,w'})) \eta(c') \prod\limits_{i=2}^n |c_i'|^{2s-n-2i+1} \, d^{\ast} c_i' = 0 \]

Taking $\eta$ to be sufficiently highly ramified to simultaneously deal  with all the $w' \in B(G)$ with $1 \leq d_B(w',e)$, we get $C_{\chi}(s, \pi_1 \eta)^{-1} - C_{\chi}(s,\pi_2  \eta)^{-1} = 0$.  This completes the proof of Proposition 3.3.3, and hence the proof of Proposition 2.5.8.  

\section*{Acknowledgements}

I would like to thank my advisor, Freydoon Shahidi, for originally recommending this problem to me, and for his many insightful ``big picture'' observations which helped me through several sticking points.  I would also like to thank James Cogdell, David Goldberg, Guy Henniart, and Chung Pang Mok for several helpful discussions on various technical points.  I would like to thank my friend Dongming She for his encouragement and support, and for his willingness to carefully check the details of my proof.

My research was supported by NSF grant DMS-1500759.

\end{document}